\documentclass[twoside,11pt,reqno]{amsart}
\usepackage{amsmath,amssymb,amscd,mathrsfs, todonotes,appendix}
\usepackage{graphics,verbatim}
\usepackage{todonotes}
\usepackage{enumitem}
\usepackage{hyperref}
\usepackage{setspace} 

\usepackage{xcolor}

\usepackage{float} 
\usepackage[labelformat=empty]{caption}
\usepackage{fancybox} 

\usepackage{extarrows}

\usepackage{quiver}

\usepackage{verbatim}

\newcommand{\losemi}{{\otimes \kern -.78em \ltimes}}
\newcommand{\rosemi}{{\otimes \kern -.78em \rtimes}}

\newcommand{\Hom}{\ensuremath{\operatorname{Hom}}}
\newcommand{\End}{\ensuremath{\operatorname{End}}}

\newcommand{\Ker}{\ensuremath{\operatorname{Ker} }}

\newcommand{\Ext}{\operatorname{Ext}}

\newcommand{\0}{\bar 0}
\newcommand{\1}{\bar 1}
\newcommand{\Z}{\mathbb{Z}}
\newcommand{\C}{\mathbb{C}}

\newcommand{\gl}{\ensuremath{\mathfrak{gl}}}
\newcommand{\g}{\ensuremath{\mathfrak{g}}}

\newcommand{\f}{\ensuremath{\mathfrak{f}}}
\newcommand{\z}{\ensuremath{\mathfrak{z}}}

\newcommand{\res}{\ensuremath{\operatorname{res}}}

\newcommand{\fg}{\ensuremath{\mathfrak{g}}}

\newcommand{\ff}{\ensuremath{\f}}

\newcommand{\fz}{\ensuremath{\mathfrak{z}}}

\newcommand{\V}{\mathcal{V}}
\newcommand{\HH}{\operatorname{H}}
\newcommand{\F}{\mathcal{F}}

\newcommand{\FF}{\mathcal{F}}

\newcommand{\XX}{\mathcal{X}}

\newcommand{\Add}{\operatorname{Add}}

\newcommand{\Proj}{\operatorname{Proj}}

\newcommand{\bT}{\mathbf T}
\newcommand{\bS}{\mathbf S}
\newcommand{\bA}{\mathbf A}

\newcommand{\bF}{\mathbf F}
\newcommand{\bK}{\mathbf K}
\newcommand{\bB}{\mathbf B}

\newcommand{\bP}{\mathbf P}
\newcommand{\bI}{\mathbf I}

\newcommand{\unit}{\ensuremath{\mathbf 1}}

\renewcommand{\mod}{\operatorname{mod}}
\newcommand{\Mod}{\operatorname{Mod}}

\newcommand{\Spc}{\operatorname{Spc}}

\newcommand{\Stab}{\operatorname{Stab}}

\newcommand{\stab}{\operatorname{stab}}
\newcommand{\Inj}{\operatorname{Inj}}

\newcommand{\supp}{\operatorname{supp}}

\newcommand{\opExt}{\operatorname{Ext}}

\makeatletter
\newcommand{\leqnomode}{\tagsleft@true}
\newcommand{\reqnomode}{\tagsleft@false}
\makeatother

\newtheorem{theorem}{Theorem}[subsection]

\makeatletter\let\c@fact\c@theorem\makeatother

\makeatletter\let\c@note\c@theorem\makeatother

\newtheorem{lemma}{Lemma}[subsection]
\makeatletter\let\c@lemma\c@theorem\makeatother

\makeatletter\let\c@lemma\c@theorem\makeatother

\makeatletter\let\c@quest\c@theorem\makeatother

\newtheorem{prop}{Proposition}[subsection]
\makeatletter\let\c@prop\c@theorem\makeatother

\newtheorem{conj}{Conjecture}[subsection]
\makeatletter\let\c@conj\c@theorem\makeatother

\makeatletter\let\c@cor\c@theorem\makeatother

\newtheorem{defn}{Definition}[subsection]
\makeatletter\let\c@defn\c@theorem\makeatother

\theoremstyle{definition}

\makeatletter\let\c@remark\c@theorem\makeatother

\makeatletter\let\c@example\c@theorem\makeatother
\numberwithin{equation}{subsection}

\usepackage[capitalise]{cleveref}
\crefname{theorem}{Theorem}{Theorems}
\crefname{fact}{Fact}{Facts}
\crefname{note}{Note}{Notes}
\crefname{lemma}{Lemma}{Lemmas}
\crefname{alg}{Algorithm}{Algorithms}
\crefname{remark}{Remark}{Remarks}
\crefname{example}{Example}{Examples}
\crefname{prop}{Proposition}{Propositions}
\crefname{conj}{Conjecture}{Conjectures}
\crefname{cor}{Corollary}{Corollaries}
\crefname{defn}{Definition}{Definitions}
\crefname{equation}{\!\!}{\!\!} 

\newcounter{listequation}

\begin{document}

\title[Homological Spectrum for Lie superalgebra representations]{The homological spectrum and nilpotence theorems for Lie superalgebra representations}

\begin{abstract}  In the study of cohomology of finite group schemes it is well known that nilpotence theorems play a key role in determining the spectrum of the cohomology ring. 
Balmer recently showed that there is a more general notion of a nilpotence theorem for tensor triangulated categories through the use of homological residue fields and the connection with the 
homological spectrum. The homological spectrum (like the theory of $\pi$-points) can be viewed as a topological space that provides an important realization of the Balmer spectrum. 

Let ${\mathfrak g}={\mathfrak g}_{\0}\oplus {\mathfrak g}_{\1}$ be a classical Lie superalgebra over ${\mathbb C}$. In this paper, the authors consider the 
tensor triangular geometry for the stable category of finite-dimensional Lie superalgebra representations: $\text{stab}({\mathcal F}_{({\mathfrak g},{\mathfrak g}_{\0})})$, 
The localizing subcategories for the detecting subalgebra ${\mathfrak f}$ are classified which answers a question of Boe, Kujawa, and Nakano. As a consequence of these results,  
the authors prove a nilpotence theorem and determine the homological spectrum for the stable module category of ${\mathcal F}_{({\mathfrak f},{\mathfrak f}_{\0})}$. The authors verify Balmer's ``Nerves of Steel'' Conjecture for ${\mathcal F}_{({\mathfrak f},{\mathfrak f}_{\0})}$.

Let $F$ (resp. $G$) be the associated supergroup (scheme) for ${\mathfrak f}$ (resp. ${\mathfrak g}$). Under the condition that $F$ is a splitting subgroup for $G$, the results for the detecting subalgebra can be 
used to prove a nilpotence theorem for $\text{stab}({\mathcal F}_{({\mathfrak g},{\mathfrak g}_{\0})})$, and to determine the homological spectrum in this case. Now using natural assumptions in terms of 
realization of supports, the authors provide a method to explicitly realize the Balmer spectrum of $\text{stab}({\mathcal F}_{({\mathfrak g},{\mathfrak g}_{\0})})$, and prove the 
Nerves of Steel Conjecture in this case. 
\end{abstract}

\author{\sc Matthew H. Hamil}
\address
{Department of Mathematics\\ University of Georgia \\
Athens\\ GA~30602, USA}
\thanks{Research of the first author was supported in part by
NSF grant DMS-2101941}
\email{matthew.hamil25@uga.edu}

\author{\sc Daniel K. Nakano}
\address
{Department of Mathematics\\ University of Georgia \\
Athens\\ GA~30602, USA}
\thanks{Research of the second author was supported in part by
NSF grant DMS-2101941}
\email{nakano@math.uga.edu}

\maketitle

\section{Introduction}

\subsection{}\label{S:homfields} For finite group representations, the detecting nilpotence in cohomology via restriction maps and elementary abelian subgroups is an important idea that was used in the study of support varieties in the work of Quillen, Avrunin and Scott. The spectrum of the cohomology for elementary abelian groups can be described through explicit realizations with polynomial rings and this yields a concrete description of the support varieties 
through a rank variety description. The cohomological nilpotence theorem plays an essential role in the theory because it allows one to describe these support varieties for a finite group with the support varieties for its elementary abelian subgroups. In the case for restricted Lie algebras, Friedlander and Parshall showed that detecting nilpotence entails the use of one-dimensional $p$-nilpotent subalgebras of the Lie algebras. Friedlander and Pevtsova developed a general theory of $\pi$-points that captures both nilpotence and the realization of supports for arbitrary finite group schemes. 

For a small rigid symmetric tensor triangulated category (TTC), $\bf K$, Balmer introduced the concept of homological primes and homological residue fields \cite{Bal20, BC21}. For a TTC, the collection of homological primes, $\text{Spc}^{\text{h}}(\bf K)$, forms a topological space that can potentially realize the Balmer spectrum, $\text{Spc}(\bf K)$,  and its support theory in a concrete way. The central problem in this identification for the theory of homological primes is the following ``Nerves of Steel'' Conjecture (cf. \cite{Bal20}).  

\begin{conj}\label{C:Nos} $\operatorname{[NoS\ Conj]}$ Let $\bf K$ be a small, rigid (symmetric) tensor triangulated category. Then the comparison map $\phi:\operatorname{Spc}^{h}(\bf K)\rightarrow \operatorname{Spc}(\bf K)$ is a bijection. 
\end{conj} 

The [NoS Conj] was verified by using a deep stratification result (see \cite{Bal20, BIKP18}) for the stable module category for finite group schemes. In this setting, the homological primes play the role of the  
$\pi$-points.

More recently, Balmer \cite{Bal20} developed a nilpotence theorem for morphisms in a tensor triangulated category and Balmer and Cameron \cite{BC21} investigated properties of homological residue fields.  Balmer showed that (i) the homological primes naturally surject via the comparison map onto the categorical spectrum and (ii) if a nilpotence theorem holds for homological residue fields then the comparison map is a bijection. In the case of finite group schemes, the picture is complete: the categorical spectrum identifies with projectivization of the cohomology ring and the homological primes with the $\pi$-points (in a non-trivial way). Recent work on the Nerves of Steel Conjecture in relation to stratification and nilpotence can be found in work by Barthel, Castellana, Heard, Naumann, Sanders, and Pol  \cite{BHS23a, BHS23b}, \cite{BCHNP23} \cite{BCHS23}.  

\subsection{} \label{S:superdetecting} Let  ${\mathfrak g}={\mathfrak g}_{\0}\oplus {\mathfrak g}_{\1}$ be a finite-dimensional classical Lie superalgebra with $\text{Lie }G_{\0}={\mathfrak g}_{\0}$ and 
$G_{\0}$ a reductive algebraic group. Moreover, let ${\mathcal F}_{({\mathfrak g},{\mathfrak g}_{\0})}$ be the category of finite-dimensional ${\mathfrak g}$-supermodules that are completely reducible over 
${\mathfrak g}_{\0}$. In the mid 2000s, Boe, Kujawa and Nakano \cite{BKN1}  introduced the notion of a detecting subalgebra for classical simple Lie superalgebras by using geometric invariant theory applied to the $G_{\0}$ action on ${\mathfrak g}_{\1}$. The detecting subalgebras are important because they detect the ${\mathcal F}_{({\mathfrak g},{\mathfrak g}_{\0})}$-cohomology, and can be regarded as the analogs of elementary abelian subgroups. 

There are two families of detecting subalgebras ${\mathfrak e}$ and ${\mathfrak f}$. The detecting subalgebras, ${\mathfrak e}$, were used to provide a geometric interpretation of the well-known combinatorial notion of 
atypicality due to Kac and Wakimoto for basic classical Lie superalgebras. On the other hand, detecting subalgebras, ${\mathfrak f}$, obtained by stable actions, were used to investigate the tensor triangular geometry and 
describe the Balmer spectrum for $\text{stab}({\mathcal F}_{({\mathfrak g},{\mathfrak g}_{\0})})$ for ${\mathfrak g}=\mathfrak{gl}(m|n)$. For any classical simple Lie superalgebra ${\mathfrak g}$, the detecting subalgebras ${\mathfrak f}$ can also be used to form a natural triangular decomposition for ${\mathfrak g}={\mathfrak n}^{+}\oplus {\mathfrak f} \oplus {\mathfrak n}^{-}$ where ${\mathfrak b}={\mathfrak f}\oplus {\mathfrak n}^{-}$ is a BBW parabolic subalgebra. The BBW parabolic subgroups/subalgebras have well-behaved homological properties as demonstrated in \cite{GGNW21}. 

\subsection{} The focus of this paper will be to study the tensor triangular geometry of the stable module category ${\mathcal F}_{({\mathfrak g},{\mathfrak g}_{\0})}$ where 
${\mathfrak g}$ is a classical Lie superalgebra. In particular, we seek to investigate the homological and Balmer spectrum for $\text{stab}({\mathcal F}_{({\mathfrak g},{\mathfrak g}_{\0})})$.
For ${\mathfrak g}=\mathfrak{gl}(m|n)$, the Balmer spectrum was computed by Boe, Kujawa and Nakano using heavy representation theoretic techniques. The approach that we use involves using a circle of ideas 
developed by Benson, Iyengar, and Krause involving the classification of localizing subcategories and by Balmer on homological primes and nilpotence for tensor triangulated categories. 

The strategy first entails classifying localizing subcategories for the detecting subalgebras. This answers a question posed in \cite{BKN4}. This result involving stratification by the cohomology ring leads to 
a nilpotence thoerem for $\text{stab}({\mathcal F}_{({\mathfrak f},{\mathfrak f}_{\0})})$. We then employ the recent work of  Serganova and Sherman \cite{SS22} on the existence of ``splitting subgroups" where one can find a copy of the trivial module in the induction of the trivial module from the splitting subgroup to the ambient supergroup to prove a nilpotence theorem for $\text{stab}({\mathcal F}_{({\mathfrak g},{\mathfrak g}_{\0})})$. 

The nilpotence theorem entails the use of homological residue fields that yield homological primes. In order to show that the homological spectrum identifies with the Balmer spectrum we employ natural assumptions on the realization of supports stated in \cite{BKN4}, which was earlier verified for ${\mathfrak g}=\mathfrak{gl}(m|n)$. This new approach has the appeal that it is much more conceptual, streamlined, and applicable to a wider class of Lie superalgebras. 

\subsection{} The paper is organized as follows. In Section~\ref{S:prelim}, a discussion about classical Lie superalgebras is presented with information about ample detecting subalgebras. Examples are provided that 
illustrate the $G_{\0}$-orbit structure on ${\mathfrak g}_{\pm 1}$ for Type I Lie superalgebras. In the following section (Section~\ref{S:Supp-Proj}) we provide conditions for projectivity for classical Lie superalgebras 
using the detecting subalgebra for finite-dimensional modules and defining the concept of ample detecting subalgebras. For our purposes, stronger detection theorems will be needed to nilpotence theorems with infinitely generated modules which will be stated using the concept of ``splitting subalgebras" as introduced in \cite{SS22}.  

Section~\ref{S:TTG} focuses on Lie superalgebras ${\mathfrak z}$ whose commutator between the even and odd component is zero, and whose even component is a torus. This class of algebras includes all detecting subalgebras. Our analysis concludes with a classification of localizing subcategories for the stable module category of ${\mathcal C}_{({\mathfrak z},{\mathfrak z}_{\bar 0})}$. 

In Section~\ref{S:Nilpotence}, we present several important results on nilpotence theorems for Lie superalgebras (cf. Theorem~\ref{T:nilpotencetwo} and Theorem~\ref{T:SupNilpotence2}). In the following section, 
Section~\ref{S:HomSpec}, the homological spectrum is calculated for Lie superalgebras that contain a splitting subalgebra isomorphic to algebras considered in Section~\ref{S:TTG}. 

Additional conditions are presented in Section~\ref{S:NoSConj}, that allow one to verify the Nerves of Steel Conjecture for these Lie superalgebras (with a splitting subalgebra). Finally, the paper concludes by providing computations of the homological spectrum and the Balmer spectrum for Type A classical Lie superalgebras (that include $\mathfrak{sl}(m|n)$) via our nilpotence theorem in Section~\ref{S:Applications}. 

\vskip .25cm
\noindent {\bf Acknowledgements.}
We acknowledge Tobias Barthel for providing references to his work with his other coauthors with connections to this paper. Moreover, we thank David Benson and Jon Carlson for comments and discussions involving Section~\ref{SS:NilColimit} . 


\section{Preliminaries}\label{S:prelim}

\subsection{Classical Lie Superalgebras} Throughout this paper, let ${\mathfrak g}$ be a Lie superalgebra over the 
complex numbers ${\mathbb C}$ where ${\mathfrak g}={\mathfrak g}_{\0}\oplus {\mathfrak g}_{\1}$ is a $\Z_{2}$-graded vector space
with a supercommutator $[\;,\;]:\g \otimes \g \longrightarrow \g$. Let $U(\g)$ be the universal enveloping superalgebra of ${\mathfrak g}$. The representation theory 
entails (super)-modules which are $\Z_{2}$-graded left $U(\g)$-modules. Given ${\mathfrak g}$-modules $M$ and $N$ 
one can use the antipode and coproduct of $U({\mathfrak g})$ to define a ${\mathfrak g}$-module
structure on the dual $M^{*}$ and the tensor product $M\otimes N$. These constructions are essential for applying the theory of tensor triangular geometry. 
For definitions and detailed background, the reader is referred to \cite{BKN1, BKN2, BKN3, BKN4}. 

A finite dimensional Lie superalgebra ${\mathfrak g}$ is called \emph{classical} if there is a connected reductive algebraic group $G_{\0}$ such
that $\operatorname{Lie}(G_{\0})=\g_{\0},$ and the action of $G_{\0}$ on $\g_{\1}$ differentiates to
the adjoint action of $\g_{\0}$ on $\g_{\1}.$  

We will be interested in studying the category of finite-dimensional and infinite-dimensional 
Lie superalgebra representations. Let ${\mathcal F}_{({\mathfrak g},{\mathfrak g}_{\0})}$ (resp. ${\mathcal C}_{({\mathfrak g},{\mathfrak g}_{\0})})$ be the full subcategory of
finite dimensional (resp. infinite-dimensional) ${\mathfrak g}$-modules which are semisimple over ${\mathfrak g}_{\0}$. A
${\mathfrak g}_{\0}$-module is \emph{semisimple} if it decomposes into a direct sum of finite dimensional simple ${\mathfrak g}_{\0}$-modules.
The categories ${\mathcal F}_{(\g,\g_{\0})}$ and  ${\mathcal C}_{({\mathfrak g},{\mathfrak g}_{\0})}$ have enough projective (and injective) objects. Furthermore,
injectivity is equivalent to projectivity \cite{BKN3} so they are Frobenius categories. Since these categories are Frobenius, one can form the 
stable module categories, $\text{stab}({\mathcal F}_{(\g,\g_{\0})})$ and  $\text{Stab}({\mathcal C}_{({\mathfrak g},{\mathfrak g}_{\0})})$.

\subsection{Detecting Subalgebras} In \cite[Section 3,4]{BKN1}, Boe, Kujawa and Nakano introduced two families of detecting subalgebras that arise for classical simple 
Lie superalgebras. One of these families, often denoted by ${\mathfrak e},$ arises from a polar action of $G_{\0}$ on ${\mathfrak g}_{\1}$. These ${\mathfrak e}$-detecting subalgebras will not be germane to the results of the paper. The other family of detecting subalgebras arises from $G_{\0}$ having a {\em stable} action on ${\mathfrak g}_{\1}$
(cf. \cite[Section 3.2]{BKN1}). Let $x_{0}$ be a generic point in ${\mathfrak g}_{\1}$ (i.e., $x_0$ is semisimple and
regular). Set $H=\text{Stab}_{G_{\0}}x_0$ and $N:=N_{G_{\0}}(H)$.
In order to construct a detecting subalgebra, let
${\mathfrak f}_{\1}={\mathfrak g}_{\1}^{H}$, ${\mathfrak f}_{\0}=[{\mathfrak f}_{\1},{\mathfrak f}_{\1}]$, and set
$${\mathfrak f}={\mathfrak f}_{\0}\oplus {\mathfrak f}_{\1}.$$
Note our definition is slightly modified from the original construction in \cite{BKN1}. 

Then ${\mathfrak f}$ is a classical Lie superalgebra and a sub Lie superalgebra
of ${\mathfrak g}$. The stability of the action of $G_{\0}$ on ${\mathfrak g}_{\1}$ 
implies the following properties.

\begin{itemize}
\item[(2.2.1)] The restriction homomorphism $S(\g_{\1}^{*}) \longrightarrow S(\f_{\1}^{*})$ induces an isomorphism
\[
\operatorname{res}: \text{H}^{\bullet}({\mathfrak g},{\mathfrak g}_{\0},{\mathbb C}) \rightarrow
\text{H}^{\bullet}({\mathfrak f},{\mathfrak f}_{\0},{\mathbb C})^{N}.
\]
\item[(2.2.2)] The set $G_{\0}\cdot {\mathfrak f}_{\1}$ is dense in $\g_{\1}.$
\end{itemize}

Property (2.2.1) will be discussed later in the context of our nilpotence theorem. For much of our work, a stronger version of property (2.2.2) is needed in order to 
prove the nilpotence theorem and compute the Balmer spectrum. 

\subsection{ Type I Lie superalgebras:} A Lie superalgebra is {\em Type I} if
it admits a $\Z$-grading ${\mathfrak g}=\g_{-1}\oplus {\g}_{0}\oplus {\g}_{1}$ concentrated
in degrees $-1,$ $0,$ and $1$ with ${\mathfrak g}_{\0}={\mathfrak g}_{0}$ and
${\mathfrak g}_{\1}=\g_{-1}\oplus {\g}_{1}$. Examples of Type I Lie superalgebras include the general linear Lie superalgebra $\mathfrak{gl}(m|n)$ and simple Lie superalgebras of
types $A(m,n)$, $C(n)$ and $P(n)$. These all have stable actions of $G_{\0}$ on ${\mathfrak g}_{\1}$ which yields Type I detecting subalgebras.

For our paper, it will be important to distinguish Type I classical Lie superalgebras that contain a detecting subalgebra with favorable geometric properties. 

\begin{defn} Let ${\mathfrak g}$ be a Type I classical Lie superalgebra with a detecting subalgebra
${\mathfrak f}={\mathfrak f}_{-1}\oplus {\mathfrak f}_{0} \oplus {\mathfrak f}_{1}$. Then ${\mathfrak f}$ is an 
{\em ample detecting subalgebra} if ${\mathfrak g}_{j}=G_{0}\cdot {\mathfrak f}_{j}$ for $j=-1,1$. 
\end{defn} 
 
As we will see in next section, there is an abundance of examples of Type I Lie superalgebras with ample detecting subalgebras that encompass many cases of simple Lie superalgebras over ${\mathbb C}$.

\subsection{Examples of  Lie Superalgebras with Ample Detecting Subalgebras} 

In this section, we will provide examples of Type I classical Lie superalgebras that contain an ample detecting subalgebra. Many of these actions involving $G_{0}$ on 
${\mathfrak g}_{\pm1}$ arise naturally in the context of linear algebra. For a more detailed description of these actions, the reader is referred to \cite[Section 3.8]{BKN3}.

\subsubsection{General Linear Superalgebra:} Let ${\mathfrak g}={\mathfrak gl}(m|n)$. As a vector space this is 
isomorphic to the set of $m+n$ by $m+n$ matrices. For a basis, one can take the elementary matrices $E_{i,j}$ where $1\leq i,j \leq m+n$. The degree zero component
${\mathfrak g}_{0}\cong \mathfrak{gl}(m)\times \mathfrak{gl}(n)$ with corresponding reductive group $G_{0}\cong GL(m)\times GL(n)$. 

Constructions of detecting subalgebras for classical Lie superalgebras are explicitly described in
\cite[Section 8]{BKN1}. Set $r=\text{min}(m,n)$. A detecting subalgebra is given by $\f=\f_{-1}\oplus \f_{\0}\oplus \f_{1}$ where
$\f_{-1}$ is the span of $\{E_{m+i,i}:\  i=1,2,\dots,r\}$, $\f_{1}$ is the span of $\{E_{i,m+i}:\  i=1,2,\dots,r\}$, and $\f_{\0}=[{\mathfrak f}_{\1},{\mathfrak f}_{1}]$. 

The action of $G_{0}$ on $\g_{-1}$ is given by $(A,B).X=BXA^{-1}$ and on $\g_{1}$ by
$(A,B).X=AXB^{-1}$. It is a well-known fact from linear algebra that the orbits representatives are the matrices of a given rank in $\g_{\pm 1}$.
It follows that ${\mathfrak g}_{\pm1}=G_{0}\cdot {\mathfrak f}_{\pm 1}$, and ${\mathfrak f}$ is an ample detecting subalgebra. 

\subsubsection{Other Type A Lie Superalgebras} The other Type A Lie superalgebras ${\mathfrak g}$ are all Type I, and 
have ${\mathfrak g}_{\pm 1}\cong {\mathfrak gl}(m|n)_{\pm 1}$. Furthermore, one has ${\mathfrak f}$ as given above for ${\mathfrak gl}(m|n)$ as a subalgebra of 
${\mathfrak g}$ \cite[Section 3.8.2 and 3.8.3]{BKN3}. 

When $m\neq n$, $\fg = \mathfrak{sl}(m|n) \subseteq \gl (m|n)$ consists of the matrices of supertrace zero, and 
\[
G_{0} = \left\{(A,B) \in GL(m) \times GL(n) \mid \det (A) \det (B)^{-1}=1 \right\},
\]
The $G_{\0}$-orbits are the same as the $GL(m) \times GL(n)$-orbits, and ${\mathfrak f}$ is an ample detecting subalgebra. 

For the Lie superalgebra $\mathfrak{sl}(n|n)$ has a one dimensional center given by scalar multiples of the identity matrix, and one has  
\[
G_{0} \cong \left\{(A,B) \in GL(n) \times GL(n) \mid \det(A)\det(B)^{-1}=1 \right\}.
\]  For elements of $\fg_{\pm 1}$ with rank strictly less than $n$, the $G_{0}$-orbits coincide with the $GL(n) \times GL(n)$-orbits.  The orbits of full rank matrices form a one parameter family 
with each orbit containing a unique matrix which is a scalar multiple of the identity. The orbit theory for the ${\mathfrak g}=\mathfrak{psl}(n|n)$ case is analogous to $\mathfrak{sl}(n|n)$. 
Consequently, in both these setting the algebra ${\mathfrak f}$ is ample.

\subsubsection{Type C Lie Superalgebras which are Type I} In this case $\fg = \mathfrak{osp}(2|2n)$ with  $G_{\0} \cong \C^{\times} \times Sp(2n)$. One has $\g_{1}\cong V_{2n}$, the natural module for $Sp(2n)$. 
The action of $Sp(2n)$ is transitive on $V_{2n}\smallsetminus\{0\}$. One has an explicit detecting subalgebra ${\mathfrak f}=\f_{-1}\oplus \f_{\0}\oplus \f_{1}$ where $\dim {\mathfrak f}_{\pm1}=1$. The transitivity of the 
action of $G_{\0}$ on ${\mathfrak g}_{1}$ shows that ${\mathfrak g}_{1}=G_{0}\cdot {\mathfrak f}_{1}$. A similar argument demonstrates that ${\mathfrak g}_{-1}=G_{0}\cdot {\mathfrak f}_{-1}$.

\subsubsection{Type P Lie Superalgebras} For Type P Lie superalgebras ${\mathfrak g}=\tilde{\mathfrak{p}}(n)$ and ${\mathfrak g}=\tilde{\mathfrak{p}}(n)$ one 
has an explicit detecting subalgebra ${\mathfrak f}=\f_{-1}\oplus \f_{\0}\oplus \f_{1}$ where ${\mathfrak f}_{\pm 1}$ contains matrices of all possible ranks. 

Let $\fg = \tilde{\mathfrak{p}}(n)$. Then $G_{\0} \cong GL(n)$ and $g_{-1} 
\cong \Lambda^{2}(V^{*})$ and $g_{1} \cong S^{2}(V)$ as $G_{\0}$-modules, where $V$ denotes the natural $GL(n)$-module. There are a finite number of 
orbits given again by the condition on rank, and their closure relation forms a chain.  This shows that ${\mathfrak f}$ is ample.

Now let $\fg = \mathfrak{p}(n) = [\tilde{\mathfrak{p}}(n),\tilde{\mathfrak{p}}(n)]$ be the simple Lie superalgebra of type $P(n-1)$.  One has $\fg_{-1}$ and $\fg_{1}$ are as above but $G_{\0} \cong SL(n)$.  
This case follows the paradigm of  $\mathfrak{sl}(n|n)$.  The $GL(n)$-orbits corresponding to matrices of rank less than $n$ in ${\mathfrak g}_{\pm1}$ are also $G_{\0}$-orbits. The matrices of rank $n$ yield a one parameter family of orbits that have orbit representatives in ${\mathfrak f}_{\pm 1}$, which demonstrates the ampleness of ${\mathfrak f}$. 

\section{Supports and Projectivity}\label{S:Supp-Proj}

\subsection{Cohomological and Rank Varieties} We review  the constructions in \cite[Section 3.2]{BKN3} for Type I Lie superalgebras. Let 
${\mathfrak g}=\fg_{-1}\oplus {\mathfrak g}_{\0} \oplus \fg_{1}$ be a Type I Lie superalgebra. Then $\fg_{\pm 1}$ are abelian Lie superalgebras. Consequently, 
$U({\mathfrak g}_{\pm1})$ identifies with an exterior algebra, and the cohomology ring for these superalgebras identifies with the 
symmetric algebra on the dual of $\fg_{\pm 1}$. Set  $R_{\pm 1}=\HH^{\bullet} (\fg_{\pm 1}, \C ) \cong S^{\bullet}(\fg_{\pm 1}^{*})$. Let $M$ be a finite-dimensional $U({\mathfrak g}_{\pm 1}))$-module and let  
\[
J_{M} = \left\{r \in R_{\pm 1} \mid r.m=0 \text{ for all } m \in \Ext_{U({\mathfrak g}_{\pm 1})}^{\bullet}(M,M) \right\}
\] 
The (cohomological) {\em support variety} of $M$ is defined as 
\[
\V_{\fg_{\pm 1}}(M) = \operatorname{MaxSpec}\left(R_{\pm 1}/J_{M} \right).
\]  
Moreover, the support variety $\V_{\fg_{\pm 1}}(M)$ is canonically isomorphic to the following rank variety: 
\[
\V_{\fg_{\pm 1}}^{\text{rank}}(M) =\left\{ x \in \fg_{\pm 1} \mid M \text{ is not projective as a $U(\langle x \rangle)$-module} \right\} \cup \{ 0 \}. 
\]
These varieties satisfy many of the important properties of support theory that include (i) the detection of projectivity over $U({\mathfrak g}_{\pm 1})$ and (ii) the tensor product property. 

For a detecting subalgebra ${\mathfrak f}={\mathfrak f}_{-1}\oplus {\mathfrak f}_{0} \oplus {\mathfrak f}_{1}$, one can apply the prior construction to obtain support varieties for 
$M\in {\mathcal F}_{({\mathfrak f},{\mathfrak f}_{0})}$, namely $\V_{{\mathfrak f}_{\pm 1}}(M)$ and $\V^{\text{rank}}_{{\mathfrak f}_{\pm 1}}(M)$. 

\subsection{Projectivity for Type I Lie Superalgebras}\label{SS:tilting} For Type I classical Lie superalgebras, one can construct Kac and dual Kac modules (cf. \cite[Section 3.1]{BKN3}). A module in 
${\mathcal F}_{({\mathfrak g},{\mathfrak g}_{0})}$ is {\em tilting} if and only if it has both a Kac and a dual Kac filtration. The use of these filtrations was a key idea in proving the following criteria for projectivity in 
the category ${\mathcal F}_{({\mathfrak g},{\mathfrak g}_{0})}$ (see \cite[Section 3]{BKN3}). 

\begin{theorem}\label{T:projective} Let ${\mathfrak g}$ be a Type I classical Lie superalgebra and $M\in {\mathcal F}_{({\mathfrak g},{\mathfrak g}_{0})}$. The following are equivalent.
\begin{itemize} 
\item[(a)] $M$ is a projective module in ${\mathcal F}_{({\mathfrak g},{\mathfrak g}_{0})}$
\item[(b)] $M$ is a tilting module 
\item[(c)] $\V_{\fg_{1}}(M) =\{0\}$ and $\V_{\fg_{-1}}(M) =\{0\}$
\item[(d)] $\V^{\operatorname{rank}}_{\fg_{1}}(M) = \{0\}$ and $\V^{\operatorname{rank}}_{\fg_{-1}}(M) =\{0\}$
\end{itemize} 
\end{theorem}

It should be noted that for an arbitrary (infinitely generated) module $M\in  {\mathcal C}_{({\mathfrak g},{\mathfrak g}_{0})}$, one can have projectivity over $U({\mathfrak g}_{\pm 1})$, but 
$M$ may not be projective in ${\mathcal C}_{({\mathfrak g},{\mathfrak g}_{0})}$. For example, if ${\mathfrak g}=\mathfrak{gl}(1|1)$, one can take an infinite coproduct of projective modules in the principal block
$P=\oplus_{m\in {\mathbb Z}} P(m|-m)$. By making suitable identifications, one can form an (infinite) ``zigzag module" (of radical length $2$) that has a Kac and dual Kac filtration, which is projective upon 
restriction to $U({\mathfrak g}_{\pm 1})$. The zigzag module is not projective in ${\mathcal C}_{({\mathfrak g},{\mathfrak g}_{0})}$ because it has radical length less than $4$. This construction can also be performed for projective modules in the principal block for the restricted enveloping algebra of ${\mathfrak sl}_{2}$ (cf. \cite{Po68}), and has been observed in other situations by Cline, Parshall and Scott \cite[(3.2) Example]{CPS85}.

\subsection{Projectivity via Ample Detecting Subalgebras} The following theorem allows us connect projectivity of a module in ${\mathcal F}_{({\mathfrak g},{\mathfrak g}_{0})}$ to 
projectivity when restricting the module to the detecting subalgebra.

\begin{theorem}\label{projectivitythm}
 Let ${\mathfrak g}$ be a Type I classical Lie superalgebra with an ample detecting subalgebra ${\mathfrak f}$ and $M\in {\mathcal F}_{({\mathfrak g},{\mathfrak g}_{0})}$. Then 
$M$ is projective in ${\mathcal F}_{({\mathfrak g},{\mathfrak g}_{0})}$ if and only if $M$ is projective in ${\mathcal F}_{({\mathfrak f},{\mathfrak f}_{0})}$.
\end{theorem}

\begin{proof} Let $M$ be projective in  ${\mathcal F}_{({\mathfrak g},{\mathfrak g}_{0})}$. Then by Theorem~\ref{T:projective}, $\V_{\fg_{\pm 1}}(M) =\{0\}$. It follows that 
$\V_{{\mathfrak f}_{\pm1}}(M) =\{0\}$ and by Theorem~\ref{T:projective}, $M$ is a projective module in ${\mathcal F}_{({\mathfrak f},{\mathfrak f}_{0})}$.  

Conversely, assume that $M$ is a projective module in ${\mathcal F}_{({\mathfrak f},{\mathfrak f}_{0})}$. Then $\V_{{\mathfrak f}_{\pm1}}(M) =\{0\}$. Let $y\in \V_{{\mathfrak g}_{1}}(M)$. 
Then $y=g\cdot x$ where $g\in G_{0}$ and $x\in {\mathfrak f}_{1}$ since ${\mathfrak f}$ is an ample detecting subalgebra. Since $M$ is a rational $G_{0}$-module, $\V_{{\mathfrak g}_{1}}(M)$ is 
$G_{0}$-stable. This implies that $x\in \V_{{\mathfrak g}_{1}}(M)$ and $x\in \V_{{\mathfrak f}_{1}}(M)$, and $x=0$. Consequently, $\V_{{\mathfrak g}_{1}}(M)=\{0\}$, and by the same reasoning 
$\V_{{\mathfrak g}_{-1}}(M)=\{0\}$. One can now conclude by Theorem~\ref{T:projective} that $M$ is a projective module in ${\mathcal F}_{({\mathfrak g},{\mathfrak g}_{0})}$.
\end{proof}

\subsection{Splitting Subalgebras} For our purposes we will need a stronger projectivity criterion than the one given in Theorem~\ref{projectivitythm} that can be applied to infinite-dimensional modules. 
In order to state such a projectivity criterion, one need to employ the concept of splitting subgroups/subalgebras as introduced by Serganova and Sherman \cite{SS22}. The following definition which fits 
in our context is equivalent to their definition. 

\begin{defn} Let ${\mathfrak g}={\mathfrak g}_{\0}\oplus {\mathfrak g}_{\1}$ be a classical Lie superalgebra and $G$ be a Lie supergroup scheme with $\operatorname{Lie }G={\mathfrak g}$. Moreover, 
let $Z\leq G$ be a subgroup with ${\mathfrak z}={\mathfrak z}_{\0}\oplus {\mathfrak z}_{\1}$ being classical and $\operatorname{Lie }Z={\mathfrak z}$. 
Then ${\mathfrak z}$ is a {\em splitting subalgebra} if and only if the trivial module ${\mathbb C}$ is a direct summand of $\operatorname{ind}_{Z}^{G} {\mathbb C}$. 
\end{defn} 

The following theorem summarizes results in \cite[Section 2]{SS22}. We present a slightly different approach using the work for BBW parabolic subgroups by D. Grantcharov, N. Grantcharov, Nakano and Wu 
\cite{GGNW21}. 

\begin{theorem}\label{T:splittingprop} Let ${\mathfrak g}$ be a classical Lie superalgebra and ${\mathfrak z}$ be a splitting subalgebra. Let $M$, $N$ be modules in ${\mathcal C}_{({\mathfrak g},{\mathfrak g}_{\0})}$.
\begin{itemize} 
\item[(a)] $R^{j}\operatorname{ind}_{Z}^{G}{\mathbb C}=0$ for $j>0$. 
\item[(b)] $M$ is projective in ${\mathcal C}_{({\mathfrak g},{\mathfrak g}_{\0})}$  if and only if $M$ when restricted to ${\mathfrak z}$ is projective in ${\mathcal C}_{({\mathfrak z},{\mathfrak z}_{\0})}$.  
\item[(c)] For all $n\geq 0$, $\opExt^{n}_{({\mathfrak g},{\mathfrak g}_{\0})}(M,N \otimes \operatorname{ind}_{Z}^{G}{\mathbb C})\cong \opExt^{n}_{({\mathfrak z},{\mathfrak z}_{\0})}(M,N)$. 
\item[(d)] For all $n\geq 0$, the restriction map $\operatorname{res}:\opExt_{({\mathfrak g},{\mathfrak g}_{\0})}^{n}(M,N)\rightarrow \opExt^{n}_{({\mathfrak z},{\mathfrak z}_{\0})}(M,N)$ is injective. 
\end{itemize} 
\end{theorem} 

\begin{proof} First observe that the category ${\mathcal C}_{({\mathfrak g},{\mathfrak g}_{\0})}$ (resp. ${\mathcal C}_{({\mathfrak z},{\mathfrak z}_{\0})}$) is equivalent to 
the rational category $\text{Mod}(G)$ (resp. $\text{Mod}(Z)$). The categories are Frobenius which means that projectivity is equivalent to injectivity. Furthermore, 
the quotient $G/Z$ is affine so the induction functor is exact (cf. \cite[Lemma 2.1]{SS22}). Thus (a) holds. 

For (b), since the quotient $G/Z$ is affine, if $M$ is projective/injective then $M$ upon restriction to $Z$ is projective/injective. Conversely, suppose that $M$ 
is a $G$-module that is injective upon restriction to $Z$. Then $\text{ind}_{Z}^{G}M$ is injective, and by the tensor identity $M\cong \text{ind}_{Z}^{G} {\mathbb C}\cong 
\text{ind}_{Z}^{G}M$. Now since $Z$ is a splitting subgroup of $G$, it follows that $M$ is a summand of $\text{ind}_{Z}^{G}M$, and is thus injective. 

For (c), there exists a spectral sequence, 
$$E_{2}^{i,j}=\opExt^{i}_{({\mathfrak g},{\mathfrak g}_{\0})}(M,N \otimes R^{j}\operatorname{ind}_{Z}^{G}{\mathbb C})\cong \opExt^{i+j}_{({\mathfrak z},{\mathfrak z}_{\0})}(M,N).$$
By using (a), the spectral sequence collapses and yields the result. For (d), one can use the isomorphism in (c), and note that the restriction map is realized by taking a summand of 
$\opExt^{n}_{({\mathfrak g},{\mathfrak g}_{\0})}(M,N \otimes \operatorname{ind}_{Z}^{G}{\mathbb C})$.  
\end{proof} 

Serganova and Sherman proved that the detecting subalgebra ${\mathfrak f}$ for classical Lie algebras of Type A are splitting subalgebras. They also present other examples of 
splitting subalgebras for Type Q classical ``simple'' Lie superalgebras that are not detecting subalgebras. See \cite[Theorem 1.1]{SS22}. An interesting question would be to 
determine for Type I classical Lie superalgebras with an ample detecting subalgebra whether the detecting subalgebra is a splitting subalgebra. 

\section{Tensor Triangular Geometry}\label{S:TTG}

\subsection{Triangulated Categories}

We will review the basic notions of triangulated categories, as these ideas lie at the foundation of the rest of the paper. Let $\mathbf{T}$ be a triangulated category. Recall that this means $\mathbf{T}$ is an additive category equipped with an auto-equivalence $\Sigma : \mathbf{T} \longrightarrow \mathbf{T}$ called the shift, and a set of distinguished triangles: 
$$
M \longrightarrow N \longrightarrow Q \longrightarrow \Sigma M
$$
all subject to a list of axioms the reader can find in \cite[Ch. 1]{Nee01}. 
\par A non-empty, full, additive subcategory $\bS$ of $\bT$ is called a \emph{triangulated subcategory} if (i) $M \in \bS$ implies that $\Sigma^n M \in \bS$ for all $n \in \mathbb{Z}$ and (ii) if $M \longrightarrow N \longrightarrow Q \longrightarrow \Sigma M$ is a distinguished triangle in $\bT,$ and two of $\{M, N, Q\}$ are objects in $\bS,$ then the third object is also in $\bS.$ A triangulated subcategory $\bS$ of $\bT$ is called \emph{thick} if $\bS$ is closed under taking direct summands. 
\par Assume that the triangulated category $\bT$ admits set-indexed coproducts. A triangulated subcategory $\bS$ of $\bT$ is called a \emph{localizing subcategory} if $\bS$ is closed under taking set-indexed coproducts. It follows from the Eilenberg swindle that localizing subcategories are necessarily thick. 
\par An object $C \in \bT$ is called \emph{compact} if $\Hom_{\bT}(C,-)$ commutes with set-indexed coproducts. The full subcategory of compact objects  in $\bT$ is denoted by $\bT^c,$ and the triangulated category $\bT$ is said to be \emph{compactly generated} if the isomorphism classes of compact objects form a set, and if for each non-zero object $M \in \bT$ there is an object $C \in \bT^c$ such that $\Hom_{\bT}(C,M) \neq 0.$

\subsection{Central Ring Actions}

Let $\mathbf{T}$ be a triangulated category with suspension $\Sigma,$ and let $R$ be a graded-commutative ring. Recall that $R$ being graded-commutative means that $R$ admits a $\mathbb{Z}$ grading and that for homogeneous elements $x, y \in R,$ $xy = (-1)^{|x||y|}yx.$ In this section we recall what it means for $R$ to act on $\mathbf{T}.$ 

Let $M$ and $N$ be objects in $\mathbf{T},$ and set  
$$
\Hom^{\bullet}_{\mathbf{T}}(M,N) := \bigoplus_{i \in \mathbb{Z}}\Hom_{\mathbf{T}}(M, \Sigma^i N). 
$$
Then $\Hom^{\bullet}_{\mathbf{T}}(M,N)$ is a graded abelian group, and $\End^{\bullet}_{\mathbf{T}}(M):= \Hom^{\bullet}_{\mathbf{T}}(M, M)$ is a graded ring where the multiplication is ``shift and compose". Notice that $\Hom^{\bullet}_{\mathbf{T}}(M,N)$ is a right $\End^{\bullet}_{\mathbf{T}}(M)$ and a left $\End^{\bullet}_{\mathbf{T}}(N)$-bimodule. 
\par The \emph{graded center} of $\mathbf{T}$ is the graded-commutative ring denoted by $Z^{\bullet}(\mathbf{T})$ whose degree $n$ component is given by 
$$
Z^n(\mathbf{T}) = \{ \eta : \text{Id}_{\mathbf{T}} \longrightarrow \Sigma^n \ | \ \eta\Sigma = (-1)^n \Sigma\eta \}.
$$
With this setup, an action of $R$ on $\mathbf{T}$ is a homomorphism $\phi: R \longrightarrow Z^{\bullet}(\mathbf{T})$ of graded rings, and if $\mathbf{T}$ admits such an action, $\mathbf{T}$ is called an $R$-linear triangulated category. It follows that if $\mathbf{T}$ is $R$-linear, then for each object $M$ in $\mathbf{T}$ there is an induced homomorphism of graded rings $\phi_M : R \longrightarrow \End^{\bullet}_{\mathbf{T}}(M)$ such that the induced $R$-module structures on $\Hom^{\bullet}_{\mathbf{T}}(M,N)$ by $\phi_M$ and $\phi_N$ agree up to sign.

\subsection{Tensor Triangulated Categories}
Our paper will concern triangulated categories that have an additional monoidal structure. Namely, we work with tensor triangulated categories as defined in \cite{Bal05}. A \emph{tensor triangulated category}, or TTC, is a triple $(\mathbf{K}, \otimes, \unit)$ where $\bK$ is a triangulated category, $\otimes: \bK \times \bK \longrightarrow \bK$ is a symmetric monoidal (tensor) product which is exact in each variable, and a monoidal unit $\unit.$
\par Let $\bK$ be a tensor triangulated category. With the additional structure of the tensor product on $\bK$ one can define analogues to the usual notions in commutative algebra like prime ideal and spectrum. A \emph{tensor ideal} in $\bK$ is a triangulated subcategory $\bI$ of $\bK$ such that $M \otimes N \in \bI$ for all $M \in \bI$ and $N \in \bK.$ A proper, thick tensor ideal $\bP$ of $\bK$ is said to be a \emph{prime ideal} if $M \otimes N \in \bP$ implies that either $M \in \bP$ or $N \in \bP.$ The \emph{Balmer spectrum} \cite[Definition 2.1]{Bal05} is defined as 
$$
\Spc(\bK) = \{\bP \subsetneq \bK \ | \ \bP \text{ is a prime ideal}\}.
$$
The Zariski topology on $\Spc(\bK)$ has as closed sets 
$$
Z(\mathcal{C}) = \{\bP \in \Spc(\bK) \ | \ \mathcal{C} \cap \bP \neq \emptyset\},
$$
where $\mathcal{C}$ is any family of objects in $\bK.$
\par The tensor triangulated category $\bK$ is said to be a \emph{compactly generated TTC} if $\bK$ is closed under set indexed coproducts, the tensor product preserves set indexed coproducts, $\bK$ is compactly generated as a triangulated category, the tensor product of compact objects is compact, $\unit$ is a compact object, and every compact object is rigid (i.e., strongly dualizable). 

\subsection{Stratifying Tensor Triangulated Categories} Given a tensor triangulated category $\bK,$ it is a difficult problem to classify its localizing subcategories and its thick tensor ideals. However, there are some circumstances in which such classifications are known, particularly in the case when $\bK$ is some category of representations. 


Stratification was introduced for triangulated categories in \cite{BIK}, and for tensor triangulated categories in \cite{BIK11}. The property of stratification allows for the classification of the localizing subcategories in terms of the $R$-support. Importantly, stratification is also easily tracked under change-of-categories, which plays a key role in these arguments. Since our work deals with tensor triangulated categories, we review that setting in this section. For a more detailed exposition the reader is referred to the original sources.

Let $\bK$ be a tensor triangulated category. Furthermore, assume that the unit object $\unit$ is a compact generator for $\bK$. Then $\End_{\bK}^{\bullet}(\unit)$ is a graded commutative ring which induces an action on $\bK$ by the homomorphisms 
$$
\End_{\bK}^{\bullet}(\unit) \xrightarrow{M \otimes -} \End_{\bK}^{\bullet}(M)
$$
for each object $M \in \bK.$ Let $R$ be a graded-commutative ring. A \emph{canonical} action of $R$ on $\bK$ is an action of $R$ on $\bK$ which is induced by a homomorphism of graded-commutative rings $R \longrightarrow \End_{\bK}^{\bullet}(\unit)$.  

Let $\Proj(R)$ denote the set of homogeneous prime ideals in $R.$ One can construct a notion of support for objects of $\bK$ by considering localization and colocalization functors. This support is described in detail in \cite[Section 3.1]{BIK}. Recall that the specialization closure of a subset $U \subseteq \Proj(R)$ is the set of all $\mathfrak{p} \in \Proj(R)$ for which there exists $\mathfrak{q} \in U$ with $\mathfrak{q} \subseteq \mathfrak{p}.$ A subset $V \subseteq \Proj(R)$ is \emph{specialization closed} if it equals its specialization closure. Given a specialization closed subset $V \subseteq \Proj(R)$, the full subcategory of $V$-torsion objects $\mathbf{T}_V $ is defined as 
$$
\mathbf{T}_V = \{M \in \mathbf{T} \ | \ \Hom_{\mathbf{T}}^{\bullet}(C,M)_{\mathfrak{p}} = 0 \ \text{for all} \ C \in \mathbf{T}^{c}, \ \mathfrak{p} \in \Proj(R) \setminus V\}.
$$
It can be checked that $\mathbf{T}_V$ is a localizing subcategory of $\mathbf{T},$ so there exists a localization functor $L_V: \mathbf{T} \longrightarrow \mathbf{T},$ and an induced colocalization functor $\mathit{\Gamma_V},$ called the \emph{local cohomology functors} of $V,$ which for each object $M \in \mathbf{T}$ give a functorial triangle
$$
\mathit{\Gamma}_VM \longrightarrow M \longrightarrow L_VM \longrightarrow .
$$

For each $\mathfrak{p} \in \Proj(R)$, let $V$ and $W$ be specialization closed subsets of $\Proj(R)$ such that $\{\mathfrak{p}\} = V \setminus W$. There is a local-cohomology functor $\mathit{\Gamma_{\mathfrak{p}}} := \mathit{\Gamma_V} L_{W}$ independent of the choice of $V$ and $W$. These functors give rise to a notion of support for objects $x \in \mathbf{T},$ and a notion of support for $\mathbf{T}$ itself. In particular, set 
$$
\supp_R(x) := \{ \mathfrak{p} \in \Proj(R) \ | \ \mathit{\Gamma_{\mathfrak{p}}}(x) \neq 0 \}, 
$$
and 
$$
\supp_R(\mathbf{T}) := \{ \mathfrak{p} \in \Proj(R) \ | \ \mathit{\Gamma_{\mathfrak{p}}}(\mathbf{T}) \neq 0 \}. 
$$
For each $\mathfrak{p} \in \Proj(R)$, the full subcategory $\mathit{\Gamma_{\mathfrak{p}}}(\bT)$ is a tensor ideal and localizing. The main definition of this section is the concept of stratification. 
\begin{defn} 
The $R$-linear tensor triangulated category $\bT$ is said to be \emph{stratified} by $R$ if $\mathit{\Gamma_{\mathfrak{p}}}(\bf T)$ is either zero or a minimal tensor ideal localizing subcategory of $\bT$ for all $\mathfrak{p} \in 
\Proj(R)$.
\end{defn}

\subsection{} For each localizing subcategory $\mathbf{C}$ of $\bT,$ set 
$$
\sigma(\mathbf{C}) = \supp_{R}(\mathbf{C}) = \{\mathfrak{p} \in \Proj(R) \ | \ \mathit{\Gamma_{\mathfrak{p}}}(\mathbf{C}) \neq 0\},
$$
and for each subset $V$ of $\Proj(R)$, set 
$$
\tau(V) = \{X \in \bT \ | \ \supp_R X \subseteq V\}.
$$ 
The following theorems \cite[Theorem 3.8]{BIK} classify localizing subcategories and thick tensor ideals via the aforementioned maps. 
\begin{theorem}
Let $\bK$ be a tensor triangulated category stratified by the action of $R.$ Then the maps $\sigma$ and $\tau$ are mutually inverse bijections between the tensor ideal localizing subcategories of $\bK$ and the subsets of $\supp_R \bK:$
$$
\{\text{tensor ideal localizing subcategories of } \bK\} \xleftrightarrow[\tau]{\sigma} \{\text{subsets of } \supp_R \bK\}.
$$
\end{theorem}

In the case when $\bK$ is stratified by a noetherian ring $R$, the maps $\sigma$ and $\tau$ also give a classification of thick tensor ideals in $\bK^c:$

\begin{theorem}
Let $\bK$ be a compactly generated tensor triangulated category which is stratified by the noetherian ring $R$.  Then the maps $\sigma$ and $\tau$ are mutually inverse bijections between the thick tensor ideal subcategories of $\bK^c$ and the specialization closed subsets of $\supp_R \bK$:
$$
\{\text{thick tensor ideal subcategories of } \bK^c\} \xleftrightarrow[\tau]{\sigma} \{\text{specialization closed subsets of } \supp_R \bK\}.
$$
\end{theorem}

\subsection{} The following result, whose proof of which can be found in \cite{Nee01}, demonstrates a close relationship between stratification of the derived category of a commutative ring with its 
canonical action, and is needed for our classifications. 

\begin{theorem}
Let $A$ be a commutative Noetherian ring and let $D(A)$ be the derived category of differential graded $A$-modules. Then $D(A)$ is stratified by the canonical action of $A.$
\end{theorem}

 \subsection{Lie Superalgebras of the Form $\mathfrak{z} = \mathfrak{z}_{\overline{0}} \oplus \mathfrak{z}_{\overline{1}}$} 

Throughout this section $\mathfrak{z} = \mathfrak{z}_{\overline{0}} \oplus \mathfrak{z}_{\overline{1}}$ denotes a Lie superalgebra where $ \mathfrak{z}_{\overline{0}} = \mathfrak{t}$ is a torus, and $[ \mathfrak{z}_{\overline{0}},  \mathfrak{z}_{\overline{1}}]=0$. This is a classical Lie superalgebra. As an example, $\mathfrak{z}$ could be any of the detecting subalgebras introduced in \cite[Section 4.4]{BKN1}. Let $\mathbf{K} = \Stab(\mathcal{C}_{ (\mathfrak{z},\mathfrak{z}_{\overline{0}})})$ be the stable module category whose objects are all $\mathfrak{z}$-modules which are finitely semi-simple as $ \mathfrak{z}_{\overline{0}}$-modules. Then $\mathbf{K}$ is a compactly generated tensor-triangulated category with full subcategory of compact objects $\mathbf{K}^{c}= \stab(\mathcal{F}_{ (\mathfrak{z},\mathfrak{z}_{\overline{0}})})$. Note that the objects of $\mathbf{K}^c$ consist precisely of the finite-dimensional objects in $\mathbf{K}$. Let $R$ denote the cohomology ring $\HH^*(\mathfrak{z},\mathfrak{z}_{\overline{0}}; \mathbb{C})$. According to \cite[Section 4.3]{BKN4}, 
$$
\HH^{\bullet}(\mathfrak{z},\mathfrak{z}_{\overline{0}}; \mathbb{C}) \cong S^{\bullet}({\mathfrak{z}_{\overline{1}}}^*)
$$
and thus $R$ is a polynomial algebra. The following theorem was established in \cite[Theorem 4.5.4]{BKN4}.  
\begin{theorem}
The thick tensor ideals in $\stab({\mathcal F}_{(\g,\g_{\0})})$ are in bijective correspondence with specialization closed subsets of $\Proj(S^{\bullet}({\mathfrak{z}_{\overline{1}}}^*)).$ Furthermore, $\Spc(\stab({\mathcal F}_{(\g,\g_{\0})}))$ is homeomorphic to $\Proj(S^{\bullet}({\mathfrak{z}_{\overline{1}}}^*))$. 
\end{theorem}

In \cite{BKN4} the authors point out that $\mathbf{K}$ is an $R$-linear triangulated category, and that the local-global principle holds. It was conjectured that $R$ stratifies $\mathbf{K},$ a result which would recover the theorem. The goal of this section is to pursue the stratification avenue, and to prove the following theorem.

\begin{theorem}
The tensor-triangulated category $\Stab(\mathcal{C}_{ (\mathfrak{z},\mathfrak{z}_{\overline{0}})})$ is stratified by the action of the cohomology ring $\HH^{\bullet}(\mathfrak{z},\mathfrak{z}_{\overline{0}}; \mathbb{C}) \cong S^{\bullet}({\mathfrak{z}_{\overline{1}}}^*)$. 
\end{theorem}

\begin{proof} 

The following argument follows closely that of \cite[Section 5.2]{BIK} where the category $\text{Stab}(kE)$ is considered, where $k$ is an algebraically closed field of characteristic two, and $E$ is an elementary abelian $2$-group. The key observation that makes these situations similar is that the cohomology rings in both cases are polynomial rings, and both the group algebra and the universal enveloping superalgebra 
of ${\mathfrak z}/{\mathfrak z}_{\0}$ are exterior algebras. 

Consider the universal enveloping superalgebra of the quotient $U(\mathfrak{z}/\mathfrak{z}_{\overline{0}})$. Then there is an isomorphism of $\mathbb{C}$-algebras $U(\mathfrak{z}/\mathfrak{z}_{\overline{0}}) \cong \Lambda^{\bullet}(\mathfrak{z}^*_{\overline{1}}).$ Therefore, there is an isomorphism of rings $\Lambda^{\bullet}(\mathfrak{z}^*_{\overline{1}}) \cong \mathbb{C}[z_1, \dots, z_r]/(z_i^2)$. Set $R=\operatorname{H}^{\bullet}({\mathfrak z},{\mathfrak z}_{\0},\mathbb C)$. Choose a basis $\{y_1, \dots, y_r\}$ of $\mathfrak{z}^*_{\overline{1}}$ so that $R \cong \mathbb{C}[y_1, \dots, y_r]$ is an isomorphism of rings, and view $R$ as a differential graded algebra with zero differential and $|y_i| = 1$ for each $i$.  

The $\mathbb{C}$-algebra $U(\mathfrak{z}/\mathfrak{z}_{\overline{0}}) \otimes_{\mathbb{C}} R$ is graded with degree $i$ component $U(\mathfrak{z}/\mathfrak{z}_{\overline{0}}) \otimes_{\mathbb{C}} R^i$ and with multiplication defined by $(a \otimes s)(b \otimes t)= ab \otimes st.$ Consider $U(\mathfrak{z}/\mathfrak{z}_{\overline{0}}) \otimes_{\mathbb{C}} R$ as a differential graded algebra with zero differential. The degree one element $\delta$ defined as 
$$
\delta = \sum_{i =1}^r z_i \otimes_{\mathbb{C}} y_i.
$$
satisfied $\delta^2=0.$ Let $J$ denote the differential graded module over $U(\mathfrak{z}/\mathfrak{z}_{\overline{0}}) \otimes_{\mathbb{C}} R$ with graded module and differential given by 
$$
J=U(\mathfrak{z}/\mathfrak{z}_{\overline{0}}) \otimes_{\mathbb{C}} R,\  d(e) = \delta e.
$$

Since $J$ is a differential graded module over $U(\mathfrak{z}/\mathfrak{z}_{\overline{0}}) \otimes_{\mathbb{C}} R,$ for each differential graded module $M$ over $U(\mathfrak{z} / \mathfrak{z}_{\overline{0}})$ there is an induced structure of a differential graded $R$-module on $\Hom_{{U}(\mathfrak{z} / \mathfrak{z}_{\overline{0}})}(J,M).$ Then the functor 
$$
\Hom_{U(\mathfrak{z}/\mathfrak{z}_{\overline{0}})}(J,-) : K(\Inj \mathcal{C}_{ (\mathfrak{z},\mathfrak{z}_{\overline{0}})}) \longrightarrow D(R)
$$
is an equivalence of triangulated categories. 

To see this first observe that as a complex, $J$ consists of injective $U(\mathfrak{z}/\mathfrak{z}_{\overline{0}})$-modules. This follows from the fact that $U(\mathfrak{z}/\mathfrak{z}_{\overline{0}})$ is self-injective. That the $R$ actions coincide follows from \cite[Theorem 5.4]{BIK}, since the R-action on $K(\Inj \mathcal{C}_{ (\mathfrak{z},\mathfrak{z}_{\overline{0}})})$ is also canonical. Finally, the result follows from the equivalence of categories $\Stab(\mathcal{C}_{ (\mathfrak{z},\mathfrak{z}_{\overline{0}})}) \cong K_{\text{ac}}(\Inj \mathcal{C}_{ (\mathfrak{z},\mathfrak{z}_{\overline{0}})})$ and the recollement from \cite[Theorem 3.19]{BIK}. 

\end{proof}

\subsection{} Theorems 4.4.2 and 4.4.3 yield the following result on localizing subcategories of thick tensor ideals for the stable module categories associated with ${\mathfrak z}$. 

\begin{theorem} Let  $\mathfrak{z} = \mathfrak{z}_{\overline{0}} \oplus \mathfrak{z}_{\overline{1}}$ denotes a Lie superalgebra where $ \mathfrak{z}_{\overline{0}} = \mathfrak{t}$ is a torus, and $[ \mathfrak{z}_{\overline{0}},  \mathfrak{z}_{\overline{1}}]=0$. 
\begin{itemize} 
\item[(a)] There is an equivalence of triangulated categories $K(\Inj \mathcal{C}_{ (\mathfrak{z},\mathfrak{z}_{\overline{0}})}) \cong D( S^{\bullet}({\mathfrak{z}_{\overline{1}}}^*))$. 
\item[(b)] The localizing subcategories of $\Stab(\mathcal{C}_{ (\mathfrak{z},\mathfrak{z}_{\overline{0}})})$ are in bijective correspondence with subsets of $\Proj(S^{\bullet}({\mathfrak{z}_{\overline{1}}}^*))$. 
\item[(c)] The thick tensor ideals in $\stab({\mathcal F}_{(\g,\g_{\0})})$ are in bijective correspondence with specialization closed subsets of $\Proj(S^{\bullet}({\mathfrak{z}_{\overline{1}}}^*))$. 
\end{itemize} 
\end{theorem}



\section{Nilpotence Theorems}\label{S:Nilpotence} 

In this section, we first recall the important ideas involving the homological spectrum, homological residue fields, and abstract nilpotence theorems in tensor triangulated categories. These concepts were developed by Balmer in \cite{Bal20}. For the ease of exposition it is convenient to modify notation slightly. Unless otherwise stated, $\bK$ will denote a tensor triangulated category satisfying the assumptions from \cite[2.2]{Bal20}, namely $\bK$ is essentially small and rigid. At times for our purposes we view $\bK$ as sitting inside of a ``large" compactly tensor triangulated category $\bT$ with $\bK = \bT^{c}.$ This is in anticipation of applying the ``Balmer machine'' to the categories $\stab(\mathcal{F}_{(\g, \g_{\overline{0}})})$ and $\Stab(\mathcal{C}_{(\g, \g_{\overline{0}})})$ which are the main point of study for this article. 

\subsection{The Homological Spectrum}\label{SS:defHomSpec} The Grothendieck abelian category of right $\bK$-modules denoted $\mathbf{A} = \Mod$-$\bK$ has as objects the additive functors $M: \bK^{\text{op}} \longrightarrow \text{Ab}$ from the opposite category of $\bK$ to the category of abelian groups, and has as morphisms the natural transformations between functors. Let $\text{h}$ denote the Yoneda embedding: 
\begin{align*}
\text{h} : \ \bK &\hookrightarrow \Mod\text{-}\bK =  \Add(\bK^{\text{op}}, \text{Ab})\\
x &\mapsto \hat{x} := \Hom_{\bK}(-,x) \\
f &\mapsto \hat{f}.
\end{align*}
For details about $\bK$-modules, the reader is referred to \cite[Section 2.6]{Bal20} and \cite[Appendix A]{BKS19}. The main facts needed in what follows are that (i) the category $\Mod$-$\bK$ admits a tensor product $\otimes: \Mod$-$\bK \longrightarrow \Mod$-$\bK$ which is right-exact in each variable and which makes $\text{h}$ a tensor functor, (ii) $\text{h}$ preserves rigidity, (iii) $\hat{x}$ is compact, projective, and $\otimes$-flat in $\bA,$ and (iv)  the tensor subcategory $\mathbf{A}^c = \mod$-$\bK$ of compact objects is abelian and is the Freyd envelope of $\bK$ (i.e., $\text{h}: \bK \longrightarrow \mod$-$\bK$ is the universal homological functor out of $\bK$.)  Recall that a functor from a triangulated category to an abelian category is {\em homological} if it maps distinguished triangles to exact sequences. Also, recall that a subcategory $\bB$ of an abelian category is called a {\em Serre subcategory} if it is closed under subobjects, quotients, and extensions. The main definitions of this subsection, which are given in \cite[Definition 3.1]{Bal20} are the following. 
\begin{defn}
\begin{itemize} 
\item [(a)] \normalfont A \emph{coherent homological prime} for $\bK$ is a maximal proper Serre $\otimes$-ideal subcategory $\bB$ of $\bA^c = \mod$-$\bK$ of the Freyd envelope of $\bK.$
\item[(b)] \normalfont The \emph{homological spectrum} of $\bK,$ denoted $\Spc^{\text{h}}(\bK)$ as a set consists of all the homological primes of $\bK:$
$$
\Spc^{\text{h}}(\bK) = \{\mathbf{B} \subsetneq \mod\text{-}\bK \ | \ \bB \text{ is maximal proper Serre ideal}\},
$$ 
and has as topology that generated by the basis of closed subsets $\supp^\text{h}(x)$ for all $x \in \bK,$ where 
$$
\supp^\text{h}(x) = \{\mathbf{B} \in \Spc^\text{h}(\bK) \ | \ \hat{x} \notin \mathbf{B} \}. 
$$
\end{itemize} 
\end{defn}



It should be noted that one can check that $(\Spc^\text{h}, \supp^\text{h})$ is a support data on $\bK.$ Therefore, there exists a unique continuous map 
$$
\phi: \Spc^\text{h}(\bK) \longrightarrow \Spc(\bK)
$$
such that $\supp^\text{h}(x) = \phi^{-1}(\supp(x))$ for all $x \in \bK.$ See \cite[Thm. 3.2]{Bal05}. This map $\phi$ is often called the \emph{comparison map}, and is surjective as long as $\bK$ is rigid.  When $\bK$ is rigid, there are many examples where the comparison map is a bijection. At the time of the writing of this article, all known examples have the property that $\phi$ is bijective. See \cite[Section 5]{Bal20}.  

\subsection{Homological Residue Fields} In this section, we recall Balmer's construction \cite{Bal20} of homological residue fields. One of the main questions in  tensor triangular geometry is to find the appropriate tensor triangular analogue to ordinary fields in commutative algebra. In particular, given $\bK,$ how does one construct functors $F: \bK \longrightarrow \bF$ to its ``residue fields''?  This question is explored in \cite{BKS19}, and some major takeaways are that there are several important properties one would like the notion of field to have. Moreover,  there are many examples of tensor triangulated categories that should be considered as tensor triangulated fields. However, it is not clear exactly what the definition should be. The following definition was proposed in \cite[Definition 1.1]{BKS19}, and will be the running definition in this work. 
\begin{defn}
\normalfont A non-trivial (big) tensor triangulated category $\bF$ is a \emph{tensor triangulated field} if every object of $\bF$ is a coproduct of compact-rigid objects of $\bF^c,$ and if every non-zero object in $\bF$ is tensor-faithful. 
\end{defn}
While this definition encapsulates many of the desired properties of fields, there is not yet a purely tensor triangular construction of them analogous to extracting residue fields in commutative algebra. Instead, Balmer uses the homological spectrum to construct homological tensor functors to abelian categories:

\begin{defn}\label{ttfield}
\normalfont Given a coherent homological prime $\mathbf{B} \in \Spc^\text{h}(\bK),$ the \emph{homological residue field} corresponding to $\mathbf{B}$ is the functor 
$$
\overline{\text{h}}_{\mathbf{B}} = Q_{\mathbf{B}} \circ \text{h} : \bK \hookrightarrow \mathbf{A} = \Mod\text{-}\bK \twoheadrightarrow \overline{\mathbf{A}}(\bK, \mathbf{B}) := \frac{\Mod\text{-}\bK}{\langle \bB \rangle}
$$
composed of the Yoneda embedding followed by the Gabriel quotient. 
\end{defn}

A natural question at this point is whether or not homological residue fields are related to the tensor triangular fields of Definition~\ref{ttfield}. The answer is yes, and an explicit connection useful for the computation of homological residue fields in examples is the content of the following theorem stated in \cite[Lemma 2.2]{BC21}.
\begin{theorem}\label{diagramthm}
Given a big tensor-triangulated category $\bT,$ a tensor-triangulated field $\bF,$ and a monoidal exact functor $F: \bT \longrightarrow \bF$ with right adjoint $U,$ one has the following diagram:


{\normalfont
\[\begin{tikzcd}
	{\mathbf{T}} & {\text{Mod-}\mathbf{T}^c} \\
	{} & {} & {\text{Mod-}\mathbf{T}^c/\Ker(\hat{F})=\overline{\mathbf{A}}_{\mathbf{B}}} \\
	{\mathbf{F}} & {\text{Mod-}\mathbf{F}^c}
	\arrow["F"', shift right=1, from=1-1, to=3-1]
	\arrow["U"', shift right=1, from=3-1, to=1-1]
	\arrow["{\text{h}}", shift left=1, from=3-1, to=3-2]
	\arrow["{\overline{F}}"', shift right=1, from=2-3, to=3-2]
	\arrow["{\text{h}}", from=1-1, to=1-2]
	\arrow["{\hat{F}}"', shift right=1, from=1-2, to=3-2]
	\arrow["{\hat{U}}"', shift right=1, from=3-2, to=1-2]
	\arrow["{\overline{U}}"', shift right=1, from=3-2, to=2-3]
	\arrow["Q"', shift right=1, from=1-2, to=2-3]
	\arrow["R"', shift right=1, from=2-3, to=1-2]
\end{tikzcd}\]
}
\vskip .15cm
\noindent
where $\hat{F}$ is the exact cocontinuous functor induced by $F,$ the functor $Q$ is the Gabriel quotient with respect to $\Ker(\hat{F})$ and the functor $\overline{F}$ is induced by the universal property, hence $\hat{F}=\overline{F}Q$ and $\overline{F}$ is exact and faithful. 

The adjunctions $F \dashv U, \hat{F} \dashv \hat{U}, \overline{F} \dashv \overline{U},$ and $Q \dashv R,$ are depicted with $\hat{F}h = hF$ and $\hat{U}h=hU$. Moreover, $\bB:=\Ker(\hat{F}) \cap \bA^{fp}$ is a homological prime and $\Ker(\hat{F}) = \langle \bB \rangle$ and $\overline{\text{h}}_{\bB}=Q \circ \text{h}: \bT \longrightarrow \overline{\bA}_{\bB}$ is a homological residue field of $\bT$.

\end{theorem}

\subsection{Nilpotence and Colimits} \label{SS:NilColimit} In this section we clarify the notions of nilpotence in the stable categories of Lie superalgebra representations and relate them to colimit constructions in module categories and homotopy colimits in the stable categories. We first discuss the concept of nilpotence. 

\begin{defn}

Let $M$ and $N$ be modules in $\mathcal{C}_{(\g, \g_{\overline{0}})}$.

\begin{itemize}

\item[(a)] A map $f: M \longrightarrow N$ is called \emph{null} if $f=0$ in $\Stab(\mathcal{C}_{(\g, \g_{\overline{0}})})$; i.e., $f$ is null if and only if $f$ factors through a projective module.

\item[(b)] A map $f: M \longrightarrow N$ is called \emph{tensor nilpotent} if there exists some $n \in \mathbb{Z}_{\geq 0}$ such that the tensor power $f^{\otimes n}: M^{\otimes n} \longrightarrow N^{\otimes n}$ is null. 

\end{itemize}
\end{defn}

In the case when $M$ is compact, one can transform the condition of the nilpotence of the map $f$ to the adjoint map. 

\begin{lemma} Let $M$ be a compact object. 
A map $f: M \longrightarrow N$ is tensor nilpotent if and only if the adjoint map $\hat{f}: k \longrightarrow M^{*} \otimes N$ is tensor nilpotent. 
\end{lemma}

\begin{proof}
By adjointness, we have an isomorphism $\Hom(M,N) \cong \Hom(k, M^{*} \otimes N)$. Since $f$ is tensor nilpotent, there exists some $n$ such that $f^{\otimes n}: M^{\otimes n} \longrightarrow N^{\otimes n}$ factors through a projective. But since tensor products of projective modules are projective, this implies that $\hat{f}^{\otimes n}$ factors through a projective; i.e., that $\hat{f}$ is tensor nilpotent. 
\end{proof}

Next we need to recall the definition of a colimit in the category $\mathcal{C}_{(\g,\g_{\0})}$ and a homotopy colimit in its stable module category. 

\begin{defn} Let ${\mathfrak g}$ be a classical Lie superalgebra. 
\begin{itemize} 
\item[(a)] Let 
$$
\theta : N_1 \overset{f_1}\longrightarrow N_2 \overset{f_2}\longrightarrow N_3 \overset{f_3}\longrightarrow \dots 
$$
be a system of modules and homomorphisms in $\mathcal{C}_{(\fg,\fg_{\0})}$. Let $\gamma: \bigoplus_{i=1}^{\infty}N_i \longrightarrow \bigoplus_{i=1}^{\infty}N_i$ be defined by $\gamma(m)=m-f_i(m)$ whenever $m \in N_i$. The \emph{colimit} of the system, if ti exists, is the module given by $\text{coker } \gamma$.
\item[(b)] Let
$$
\theta : X_1 \overset{f_1}\longrightarrow  X_2 \overset{f_2}\longrightarrow X_3 \overset{f_3}\longrightarrow \dots 
$$
be a system of modules and homomorphisms in $\Stab(\mathcal{C}_{(\g, \g_{\overline{0}})})$. The \emph{homotopy colimit} of the system is the module obtained by completing the map 
$$
\bigoplus X_i \overset{1-f}\longrightarrow \bigoplus X_i
$$

to a triangle: 

$$
\bigoplus X_i \overset{1-f}\longrightarrow \bigoplus X_i \longrightarrow \text{hocolim}(X_i) \longrightarrow . 
$$
\end{itemize} 
\end{defn}

The following lemmas are given in \cite{Ric97} and are modified here for Lie superalgebra representations. 

\begin{lemma}\label{lemma1}
Let $X_1 \overset{\alpha_1}\longrightarrow X_2 \overset{\alpha_2}\longrightarrow \cdots$ be a sequence of maps in a triangulated category with countable direct sums. If for each $i>0$ there exists $k>i$ such that $\alpha_1\dots\alpha_k = 0$, then $\text{hocolim}(X_i) \cong 0$.
\end{lemma}

The next lemma clarifies the relationship between homotopy colimits in the stable category with colimits in the ordinary module category. 
\begin{lemma}\label{lemma2}
Let $X_1 \overset{\alpha_1}\longrightarrow X_2 \overset{\alpha_2}\longrightarrow \cdots$ be a sequence of modules and homomorphisms in $\Stab(\mathcal{C}_{(\g, \g_{\overline{0}})})$. The colimit $\text{colim}(X_i)$ in $\mathcal{C}_{(\g, \g_{\overline{0}})}$ is isomorphic in $\Stab(\mathcal{C}_{(\g, \g_{\overline{0}})})$ to the homotopy colimit $\text{hocolim}(X_i)$. 
\end{lemma}

These two lemmas together allow one to derive an analog of \cite[Lemma 2.3]{BC18}

\begin{theorem} \label{T:Rickard} 
A map $f: k \longrightarrow N$ is $\otimes$-nilpotent if and only if the colimit of 
$$
k \overset{f}\longrightarrow N \overset{f \otimes 1}\longrightarrow N \otimes N \overset{f \otimes 1 \otimes 1}\longrightarrow N \otimes N \otimes N \longrightarrow \cdots
$$
is projective. 
\end{theorem}

\begin{proof}
First suppose that $f: k \longrightarrow N$ is $\otimes$-nilpotent. Then Lemma~\ref{lemma1} implies that the homotopy colimit of this system viewed in the stable category is isomorphic to zero, which is to say the colimit of the system is zero by Lemma~\ref{lemma2}. Now suppose that the colimit of the system is projective. Again, when viewed in the stable category this implies that the homotopy colimit is zero, which gives the tensor nilpotence of $f$.
\end{proof}

\subsection{Nilpotence Theorems} Nilpotence theorems have played an important role in cohomology and representation theory. Devinatz, Hopkins, and Smith showed in \cite{DHS88} that a map between finite spectra which gets annihilated by all Morava $K$-theories must be tensor-nilpotent. Neeman \cite{Nee92} and Thompson \cite{Tho97} proved nilpotence theorems for maps in derived categories using ordinary residue fields, and Benson, Carlson, and Rickard \cite{BCR97} proved nilpotence theorems in modular representation theory, where the residue fields are given by cyclic shifted subgroups, or, in the case of finite group schemes, $\pi$-points \cite{FP07}. With these examples in mind, Balmer in \cite{Bal20} using homological residue fields.presented a unified treatment that applies to all tensor triangulated categories. In the case where $\bK$ sits inside of a big rigidly compactly generated tensor triangulated category $\bT$ with $\bK = \bT^c,$ one can make a connection to the homological spectrum. In particular, he proved the following theorem \cite[Corollary 4.7]{Bal20}: 


\begin{theorem}\label{abstractnilpotencebig}
Let $\bT$ be a rigidly-compactly generated ``big" tensor-triangulated category with $\bK = \bT^c.$ Let $f: x \longrightarrow Y$ be a morphism in $\bT$ with $x \in \bK$ compact and $Y$ arbitrary. If $\overline{\emph{h}}(f)=0$ in $\overline{\bA}(\bK;\bB)$ for every homological residue field $\overline{\emph{h}}_{\bB}$ for every homological prime $\bB \subsetneq \mod$-$\bK,$ then there exists $n \geq 1$ such that $f^{\otimes n}=0$ in $\bT.$
\end{theorem}

The nilpotence theorem stated above can combined with the theory of detecting subalgebras developed by Boe, Kujawa, and Nakano, to the study of nilpotence in the stable categories of Lie superalgebra representations. The following nilpotence theorem via homological residue fields is a direct translation of Theorem~\ref{abstractnilpotencebig} in the context of superalgebra representatiions.



\begin{theorem}\label{T:nilpotencetwo}
Let $\fg = \fg_{\0} \oplus \fg_{\1}$ be a classical Lie superalgebra, and let $f: M \longrightarrow N$ be a morphism in $\Stab(\mathcal{C}_{(\fg, \fg_{\overline{0}})})$ with $M \in \stab(\mathcal{F}_{(\fg, \fg_{\overline{0}})}).$ Suppose that $\overline{\emph{h}}_{\bB}(f)=0$ for all $\bB \in \Spc^{\emph{h}}(\stab(\mathcal{F}_{(\fg, \fg_{\overline{0}})})).$ Then there exists $n \geq 1$ such that $f^{\otimes n}= 0$ in $\Stab(\mathcal{C}_{(\fg, \fg_{\overline{0}})}).$
\end{theorem}

\begin{proof}
The first observation is that $\stab(\mathcal{F}_{(\fg, \fg_{\overline{0}})})$ sits inside of $\Stab(\mathcal{C}_{(\fg, \fg_{\overline{0}})})$ as the compact objects: $\stab(\mathcal{F}_{(\fg, \fg_{\overline{0}})}) = (\Stab(\mathcal{C}_{(\fg, \fg_{\overline{0}})}))^c.$ Moreover, the compact objects and the rigid objects coincide and generate $\Stab(\mathcal{C}_{(\fg, \fg_{\overline{0}})})$ as a tensor-triangulated category. This is the setup of Theorem~\ref{abstractnilpotencebig}. 
\end{proof}

\subsection{A Nilpotence Theorem via Detecting Subalgebra} The salient feature first discovered about detecting subalgebras was that these subalgebras detect nilpotence in cohomology. We will now show that 
a remarkable feature for classical Lie subalgebras with a splitting subalgebras is that nilpotence of arbitrary maps in the stable module category is governed by nilpotence when restricting the 
the map to a splitting subalgebra. In particular, to show that a morphism $f: M \longrightarrow N$ is nilpotent in the big stable module category where $M$ is compact, it is enough to check vanishing on those homological residue fields constructed via homological primes from the stable categories of modules over the splitting subalgebra. 



\begin{theorem}\label{T:SupNilpotence2}
Let $\fg = \fg_{\0} \oplus \fg_{\1}$ be a classical Lie subalgebra with a splitting subalgebra $\fz = \fz_{\0} \oplus \fz_{\1} \subseteq \fg.$ Let $f: M \longrightarrow N$ be a morphism in $\Stab(\mathcal{C}_{(\fg, \fg_{\overline{0}})})$ 
with $M\in \stab(\mathcal{F}_{(\fg, \fg_{\overline{0}})})$. If $\overline{\emph{h}}_{\bB}(\operatorname{res}(f))=0$ for all $\bB \in \Spc^{\emph{h}}(\stab(\mathcal{F}_{(\fz, \fz_{\overline{0}})})$ then there exists $n \geq 1$ such that $f^{\otimes n}= 0$ in $\Stab(\mathcal{C}_{(\fg, \fg_{\overline{0}})}).$
\end{theorem}

\begin{proof}
Let $\text{res} : \Stab(\mathcal{C}_{(\fg, \fg_{\overline{0}})}) \longrightarrow \Stab(\mathcal{C}_{(\fz, \fz_{\overline{0}})})$ be the usual restriction functor. By our hypothesis, 
$\overline{\text{h}}_{\bB}(\res(f)) = 0$ for all $\bB \in \Spc^{\emph{h}}(\stab(\mathcal{F}_{(\fz, \fz_{\overline{0}})})$, Theorem~\ref{T:nilpotencetwo} implies that $\res(f)$ is tensor nilpotent in $\Stab(\mathcal{C}_{(\fz, \fz_{\overline{0}})})$. 

It follows that $\text{res}(\widehat{f}):{\mathbb C}\rightarrow M^{*}\otimes N$ is tensor nilpotent in  $\Stab(\mathcal{C}_{(\fz, \fz_{\overline{0}})})$, and by Theorem~\ref{T:Rickard} its associated colimit is projective in $\mathcal{C}_{({\mathfrak z},{\mathfrak z}_{\0})}$. Therefore, the colimit as an object in $\mathcal{C}_{({\mathfrak g},{\mathfrak g}_{\0})}$ is projective by Theorem~\ref{T:splittingprop}(b). Invoking Theorem~\ref{T:Rickard} again implies that $\widehat{f}$ is tensor nilpotent, thus $f$ is tensor nilpotent.  
\end{proof}

\section{Identifying the Homological Spectrum}\label{S:HomSpec}

The goals of this section are to determine the homological spectrum for $\stab(\mathcal{F}_{(\fz, \fz_{\overline{0})}})$ and $\stab(\mathcal{F}_{(\fg, \fg_{\overline{0}})}),$ where $\fg$ is a classical Lie superalgebra with splitting subalgebra $\fz.$ We also consider the comparison map first defined in Section~\ref{SS:defHomSpec}. 

\subsection{Stratification and the Comparison Map}
The first observation in this direction is that classification of localizing subcategories via stratification enables one to show that the comparison map is a bijection: 

\begin{theorem}\label{T:zstratification} 
Let $\fz = \fz_{\0} \oplus \fz_{\1}$ be a Type I classical Lie superalgebra with $\fz_{\0}$ a torus and $[\fz_{\0}, \fz_{\1}] = 0.$ Then the comparison map 
$$
\phi :\Spc^h(\stab(\mathcal{F}_{(\fz, \fz_{\overline{0}})})\longrightarrow \Spc(\stab(\mathcal{F}_{(\fz, \fz_{\overline{0}})}))
$$
is a bijection.
\end{theorem}
\begin{proof} Since $\mathcal{F}_{({\mathfrak z},{\mathfrak z}_{\0})}$ is rigid, the map $\phi$ is surjective. In order to prove that the map is injective one can use the argument outlined 
in \cite[Example 5.13]{Bal20}. The main point is to use the classification of localizing subcategories of $\mathcal{C}_{({\mathfrak z},{\mathfrak z}_{\0})}$ and the existence of pure injective objects. 
See also \cite[Corollary 4.26]{BKS19}. 
\end{proof}

A more general argument that shows that stratification implies the Nerves of Steel Conjecture can be found in \cite[Theorem 4.7]{BHS23a}. 

\subsection{} Let ${\mathfrak g}$ be a classical Lie superalgebra and ${\mathfrak z}$ be a detecting subalgebra in ${\mathfrak g}$. We will need to work with a field extension $K$ of ${\mathbb C}$ such that the 
transcendence degree is larger than the dimension of ${\mathfrak z}$. Note that this is the analogous setup as in \cite[Example 3.9]{BC21}. The stable module categories involved will be viewed over the field extension 
$K$. Let ${\mathcal P}_{x}$ be the prime ideal in $\text{Proj}(S^{\bullet}({\mathfrak z}_{\1}))$ associated with the ``generic point" $x$ (cf. \cite[Sections 2 and 3] {BCR96} for an explanation). 

For $x\in {\mathfrak z}_{\1}$ with ${\mathfrak z}_{\1}$ viewed as a vector space over $K$, one has 
$U(\langle x. \rangle )$ is either ${\mathbb K}[x]/(x^{2})$ or $U({\mathfrak q}(1))$. In either case, the blocks are either semisimple or have finite representation type. One can verify that 
$\text{Stab}(C_{(\langle x \rangle, \langle x \rangle_{\0})})$ is a tensor triangular field. For $x\in {\mathfrak z}_{\1}$, one have two monoidal exact functors (given by restriction): 
\begin{equation}
\pi_{x}^{\mathfrak g}: \text{Stab}(C_{({\mathfrak g},{\mathfrak g}_{\0})})\rightarrow \text{Stab}(C_{(\langle x \rangle, \langle x \rangle_{\0})})
\end{equation} 
\begin{equation}
\pi_{x}^{\mathfrak z}: \text{Stab}(C_{({\mathfrak z},{\mathfrak z}_{\0})})\rightarrow \text{Stab}(C_{(\langle x \rangle, \langle x \rangle_{\0})})
\end{equation} 

Let $\text{res}: \text{Stab}(C_{({\mathfrak g},{\mathfrak g}_{\0})})\rightarrow \text{Stab}(C_{({\mathfrak z},{\mathfrak z}_{\0})})$ be the natural functor obtained by restricting ${\mathfrak g}$-modules to 
${\mathfrak z}$-modules. Then $\pi_{x}^{\mathfrak g}=\pi_{x}^{\mathfrak z}\circ \text{res}$ for all $x\in {\mathfrak z}_{\1}$. 

Now one can apply Theorem~\ref{diagramthm} (where $F=\pi_{x}^{\mathfrak g}$ and $\pi_{x}^{\mathfrak z}$), to obtain ${\mathcal B}_{x}$ a homological prime (resp. ${\mathcal B}_{x}^{\prime}$) associated to 
$\pi_{x}^{\mathfrak g}$ (resp. $\pi_{x}^{\mathfrak z}$). Similarly, let $\overline{h}_{{\mathcal B}_{x}}$ (resp. $\overline{h}_{{\mathcal B}^{\prime}_{x}}$) be the homological residue field corresponding to ${\mathcal B}_{x}$ (resp. ${\mathcal B}^{\prime}_{x}$). 

\subsection{} Let $\fg = \fg_{\0} \oplus \fg_{\1}$ be a Type I classical Lie superalgebra with detecting subalgebra $\fz = \fz_{\0} \oplus \fz_{\1} \subseteq \fg.$ Let $x \in \fz_{\1},$ and let $\pi_x : \langle x \rangle \longrightarrow \fg$ be the usual inclusion of the Lie subsuperalgebra generated by $x$ into $\fg.$ This yields, for each $x \in \fz_{\1},$ an induced functor $\pi_x^* : \Stab(\mathcal{C}_{(\fg, \fg_{\overline{0}})}) \longrightarrow \Stab(\mathcal{C}_{(\langle x \rangle, {\langle x \rangle}_{\overline{0}})}).$ These are monoidal functors, each of which with right adjoint given by induction. Moreover, $\Stab(\mathcal{C}_{(\langle x \rangle, {\langle x \rangle}_{\0})})$ is a tensor-triangulated field, so Theorem~\ref{diagramthm} produces a homological residue field. Denote the corresponding homological prime of $\stab(\mathcal{F}_{(\fg, \fg_{\overline{0})}})$ by $\bB_x.$ The goal now is the show that $\{\bB_x\}_{x \in \fz_{\1}}$ contains all of the homological primes. Recall the following result in \cite[Theorem 5.4]{Bal20}. 

\begin{theorem}\label{T:classify}
Let $\bT$ be a big tensor-triangulated category with $\bK = \bT^c.$ Consider a family ${\mathcal E} \subseteq \Spc^h(\bK)$ of points in the homological spectrum. Suppose that the corresponding functors $\overline{h}_{\bB}: \bT \longrightarrow \overline{\bA}(\bK; \bB)$ collectively detect $\otimes$-nilpotence in the following sense: If $f: x \longrightarrow Y$ in $\bT$ is such that $x \in \bT^c$ and $\overline{h}_{\bB}(f) = 0$ for all $\bB \in {\mathcal E}$ then $f^{\otimes n} = 0$ for some $n\geq 1.$ Then we have ${\mathcal E} = \Spc^h(\bK).$
\end{theorem}

\subsection{} We are now ready to provide conditions on when one can identify a collection of homological primes that detect nilpotence on  $\text{stab}({\mathcal F}_{({\mathfrak g},{\mathfrak g}_{\0})})$. 

\begin{theorem} Let ${\mathfrak g}={\mathfrak g}_{\0}\oplus {\mathfrak g}_{\1}$ be a classical Lie superalgebra and ${\mathfrak z} \leq {\mathfrak g}$ be a sub Lie superalgebras. Denote by 
$G$, $G_{\0}$ and $Z$ the associated supergroup (schemes) such that ${\mathfrak g}=\operatorname{Lie }G$, ${\mathfrak g}_{\0}=\operatorname{Lie }G_{\0}$ and ${\mathfrak z}=\operatorname{Lie }Z$. Set $N=N_{G_{\0}}({\mathfrak z}_{\1})$. Assume that 
\begin{itemize} 
\item[(a)] ${\mathfrak z}={\mathfrak z}_{\0}\oplus {\mathfrak z}_{\1}$ with $[{\mathfrak z}_{\0},{\mathfrak z}_{\1}]=0$; 
\item[(b)] $Z$ is a splitting subgroup of $G$. 
\end{itemize} 
Then ${\mathcal E}/N=\{{\mathcal B}_{x}: x \in {\mathfrak z}_{\1}\}/N$ (i.e., a set of $N$-orbit representatives) detects nilpotence in $\operatorname{stab}({\mathcal F}_{({\mathfrak g},{\mathfrak g}_{\0})})$. 
\end{theorem} 

\begin{proof} The idea of the proof is to find a set of homological primes ${\mathcal E}$ that detects nilpotence in $\text{Stab}(\mathcal{C}_{(\fg, \fg_{\overline{0}})})$. Then one can apply 
Theorem~\ref{T:classify} (e.g., \cite[Theorem 5.4]{Bal20}.) 

The first step is to compare homological residue fields for ${\mathfrak g}$ and ${\mathfrak z}$. If $f:M\rightarrow N$ is in $\text{Stab}(C_{({\mathfrak g},{\mathfrak g}_{\0})})$ with $M$ compact 
then one can compare the diagrams for 
$\overline{h}_{{\mathcal B}_{x}}$ and $\overline{h}_{{\mathcal B}^{\prime}_{x}}$ to conclude the following. 
\vskip .15cm 
\noindent 
(1) {\em If $\overline{h}_{{\mathcal B}_{x}}(f)=0$ then $\overline{h}_{{\mathcal B}^{\prime}_{x}}(\operatorname{res}(f))=0$ for $x\in {\mathfrak z}_{\1}$.} 
\vskip .15cm 
Now one can apply the stratification result, Theorem~\ref{T:zstratification}, to conclude that $\{{\mathcal B}^{\prime}_{x}:\ x\in {\mathfrak z}_{\1}\}$ are the homological primes for 
$\operatorname{Stab}(\mathcal{F}_{(\fz, \fz_{\overline{0}})})$. Therefore, by (1) and Theorem~\ref{T:nilpotencetwo}, one has 
\vskip .15cm 
\noindent
(2) {\em If $\overline{h}_{{\mathcal B}^{\prime}_{x}}(\operatorname{res}(f))=0$ for all $x\in {\mathfrak z}_{\1}$ then $\operatorname{res}(f):M\rightarrow N$ is $\otimes$-nilpotent in 
$\operatorname{Stab}(\mathcal{C}_{(\fz, \fz_{\overline{0}})})$.}
\vskip .15cm
Applying Theorem~\ref{T:SupNilpotence2} since ${\mathfrak z}$ is a splitting subalgebra of ${\mathfrak g}$, one can conclude that 
$f:M\rightarrow N$ is $\otimes$-nilpotent in $\text{Stab}(\mathcal{C}_{(\fg, \fg_{\overline{0}})})$. Let ${\mathcal E}=\{{\mathcal B}_{x}: x \in {\mathfrak z}_{\1}\}/N$. Since $M$ is a 
$G_{\0}$-module, it follows that the functors $\pi^{\mathfrak g}_{x}$ (resp. $\pi^{\mathfrak g}_{nx}$) will provide the same decomposition of $M$ in $\text{Stab}(C_{(\langle x \rangle, \langle x \rangle_{\0})})$ 
(resp. $\text{Stab}(C_{(\langle nx \rangle, \langle nx \rangle_{\0})})$). By considering Theorem~\ref{diagramthm}, it follows that $\overline{h}_{{\mathcal B}_{x}}(f)=0$ if and only if 
$\overline{h}_{{\mathcal B}_{nx}}(f)=0$. Therefore, ${\mathcal E}/N$ detects nilpotence. 

\end{proof} 

In the previous theorem, one can state that ${\mathcal E}/N=\text{Spc}^{h}({\mathcal F}_{({\mathfrak g},{\mathfrak g}_{\0})})$. However, with the definition of ${\mathcal E}/N$ there are certain homological 
primes that might be identified in the set. We will show in the following section that different $N$-orbit representatives yield different elements in 
$\operatorname{Spc}^{h}(\text{stab}({\mathcal F}_{({\mathfrak g},{\mathfrak g}_{\0})}))$. 

\section{Nerves of Steel Conjecture}\label{S:NoSConj}

\subsection{} There are noticeable differences between the stable module category for finite group schemes versus the stable module category for Lie superalgebras. For example, 
the comparison map: 
\begin{equation} 
\operatorname{Spc}(\text{stab}({\mathcal F}_{({\mathfrak g},{\mathfrak g}_{\0})}))\rightarrow \operatorname{Proj}(\operatorname{H}^{\bullet}({\mathfrak g},{\mathfrak g}_{\0},{\mathbb C}))
\end{equation} 
is not always homeomorphism (e.g., when ${\mathfrak g}=\mathfrak{gl}(m|n)$).  In general the cohomology ring $\operatorname{H}^{\bullet}({\mathfrak g},{\mathfrak g}_{\0},{\mathbb C})$ does not stratify $\text{Stab}({\mathcal C}_{({\mathfrak g},{\mathfrak g}_{\0})})$. There are many examples  where the support theory does not detect projectivity. This is the main reason one needs to use the cohomology of the detecting subalgebra to realize the homological spectrum and the Balmer spectrum. 

Boe, Kujawa and Nakano \cite{BKN4} showed that for ${\mathfrak g}=\mathfrak{gl}(m|n)$, one has a homeomorphism: 
\begin{equation} 
\operatorname{Spc}(\text{stab}({\mathcal F}_{({\mathfrak g},{\mathfrak g}_{\0})}))\cong \text{$N$-}\operatorname{Proj}(\operatorname{H}^{\bullet}({\mathfrak f},{\mathfrak f}_{\0},{\mathbb C}))
\end{equation} 
where ${\mathfrak f}$ is a detecting (splitting) subalgebra of ${\mathfrak g}$ and $N$ is the normalizer of ${\mathfrak f}_{\1}$ in $G_{\0}$. From this example, it is clear that in order to compute the Balmer spectrum for Lie superalgebras one needs to find a suitable replacement for the cohomology ring. 

From Section~\ref{S:HomSpec}, when one has a splitting subalgebra ${\mathfrak z}$ of ${\mathfrak g}$, one can compute the homological spectrum and show there is a surjection: 
\begin{equation} 
\operatorname{Spc}^{h}(\text{stab}({\mathcal F}_{({\mathfrak g},{\mathfrak g}_{\0})}))\rightarrow \text{$N$-}\operatorname{Proj}(\operatorname{H}^{\bullet}({\mathfrak z},{\mathfrak z}_{\0},{\mathbb C}))
\end{equation} 
Since ${\mathcal F}_{({\mathfrak g},{\mathfrak g}_{\0})}$ is rigid, the comparison map
\begin{equation} 
\phi:\operatorname{Spc}^{h}(\text{stab}({\mathcal F}_{({\mathfrak g},{\mathfrak g}_{\0})}))\rightarrow \operatorname{Spc}(\text{stab}({\mathcal F}_{({\mathfrak g},{\mathfrak g}_{\0})}))
\end{equation} 
is surjective. Our goal is to use the prior calculation of the homological spectrum to give conditions on when the Nerves of Steel Conjecture holds (i.e., when $\phi$ is bijective). 

\subsection{} We can now identify the homological spectrum and the Balmer spectrum for classical Lie superalgebras with a splitting subalgebra under a suitable realization condition. 

\begin{theorem}\label{T:capstone} Let ${\mathfrak g}$ be a classical Lie superalgebra with a splitting subalgebra ${\mathfrak z}\cong {\mathfrak z}_{\0}\oplus {\mathfrak z}_{\1}$. Assume that 
\begin{itemize} 
\item[(i)] ${\mathfrak z}={\mathfrak z}_{\0}\oplus {\mathfrak z}_{\1}$ where ${\mathfrak z}_{\0}$ is a torus and $[{\mathfrak z}_{\0},{\mathfrak z}_{\1}]=0$.
\item[(ii)] Given $W$ an $N$-invariant closed subvariety of $\Proj (S^{\bullet}(\z_{\1}^*))$,  there exists $M \in \stab(\mathcal{F}_{(\g, \g_{\0})})$ with $V_{(\z, \z_{\0})}(M) = W$. 
\end{itemize} 
Then  
\begin{itemize}
\item[(a)] There exists a 1-1 correspondence 
$$
\{\text{thick tensor ideals of $\operatorname{stab}({\mathcal F}_{({\mathfrak g},{\mathfrak g}_{\0})})$}\} \begin{array}{c} {} \atop {\longrightarrow} \\ {\longleftarrow}\atop{} \end{array}  \XX_{sp}
$$
where $X=N\text{-}\operatorname{Proj}(S^{\bullet}({\mathfrak z}_{\1}))$ and ${\mathcal X}_{sp}$ are specialization closed sets of $X$. 
\item[(b)] There exists a homeomorphism $\eta: N$-$\operatorname{Proj}(S^{\bullet}({\mathfrak z}_{\1}))\longrightarrow \Spc(\operatorname{stab}{(\mathcal F}_{({\mathfrak g},{\mathfrak g}_{\0})}))$. 
\item[(c)] The comparison map $\phi:\operatorname{Spc}^{h}(\operatorname{stab}({\mathcal F}_{({\mathfrak g},{\mathfrak g}_{\0})}))\rightarrow 
\operatorname{Spc}(\operatorname{stab}({\mathcal F}_{({\mathfrak g},{\mathfrak g}_{\0})}))$ is bijective. 
\end{itemize}
\end{theorem}

\begin{proof} (a) and (b) follow by \cite[Theorems 3.4.1, 3.5.1]{BKN4}. For part (c), let $\rho=\eta^{-1}$ which is given by a concrete description in \cite[Corollary 6.2.4]{NVY}. Consider the following diagram of 
topological spaces: 
\begin{equation}
\label{diag-M-g}
\begin{tikzcd}
\operatorname{Spc}^{h}(\operatorname{stab}({\mathcal F}_{({\mathfrak z},{\mathfrak z}_{\0})}))  \arrow[r, "\phi^{\prime}"]  \arrow[d, "\theta"] & \operatorname{Spc}(\operatorname{stab}({\mathcal F}_{({\mathfrak z},{\mathfrak z}_{\0})}))  
\arrow[r, "\rho^{\prime}"]    \arrow[d, "\hat{\pi}"]      &  \operatorname{Proj}(S^{\bullet}({\mathfrak z}_{\1}))   \arrow[d, "\pi"]   \\
\operatorname{Spc}^{h}(\operatorname{stab}({\mathcal F}_{({\mathfrak g},{\mathfrak g}_{\0})}))  \arrow[r, "\phi"] & \operatorname{Spc}(\operatorname{stab}({\mathcal F}_{({\mathfrak g},{\mathfrak g}_{\0})}))  
\arrow[r, "\rho"]         &  N\text{-}\operatorname{Proj}(S^{\bullet}({\mathfrak z}_{\1}))   \\
\end{tikzcd}
\end{equation}
One has that $\rho^{\prime}$ is a homeomorphism and $\phi^{\prime}$ is a bijection for ${\mathfrak z}$. From part (b), the map $\rho$ is a homomorphism. The maps $\pi$ and $\hat{\pi}$ are surjections. The map 
$\theta$ sends ${\mathcal B}_{x}$ to ${\mathcal B}^{\prime}_{x}$ in ${\mathcal E}/N$. Suppose that $\phi({\mathcal B}_{x_{1}})=\phi({\mathcal B}_{x_{2}})$. The using the commutativity, one has ${\mathcal P}_{x_{1}}={\mathcal P}_{x_{2}}$ in $N$-$\operatorname{Proj}(S^{\bullet}({\mathfrak z}_{\1}))$ which means that 
$x_{1}$ and $x_{2}$ are $N$-conjugate. This proves that ${\mathcal B}_{x_{1}}={\mathcal B}_{x_{2}}$ in ${\mathcal E}/N$.

\end{proof}  

We remark that the verification of the Nerves of Steel Conjecture in the previous theorem uses stratification results only for $\text{Stab}({\mathcal C}_{({\mathfrak z},{\mathfrak z}_{\0})})$, unlike the 
the case for finite group schemes where the stratification is needed for $\text{Stab}(G)$ (see \cite[5.13 Example]{Bal20}). A general stratification result for $\text{Stab}({\mathcal C}_{({\mathfrak g},{\mathfrak g}_{\0})})$ will 
be addressed in future work.

\section{Applications}\label{S:Applications} 

\subsection{}\label{SS:varietynotations}  We need to verify the realization condition in Theorem~\ref{T:capstone}(ii) for other Lie superalgebras of Type A, namely ${\mathfrak g}=\mathfrak{sl}(m|n)$. The main ideas have been established in the $\mathfrak{gl}(m|n)$-case (cf \cite[Section 7 ]{BKN4}). We will use the same notation to show how the arguments need to modified to handle the $\mathfrak{sl}(m|n)$-case.

Let $\fg =\mathfrak{sl}(m|n)$ and write $\ff$ for the detecting subalgebra of $\fg$. Without loss of generality one can assume that $m\leq n$.  The subalgebra $\ff_{\1}$ has a basis given by the matrix units $$e_{1, 2m}, e_{2, 2m-1}, \dotsc , e_{m, m+1}, e_{m+1,m}, \dotsc , e_{2m,1}.$$    Let $T$ denote the torus of $G_{\0}$ consisting of diagonal matrices.   The torus $T$ acts on ${\mathfrak f}$ by the adjoint action and the aforementioned basis for $\ff_{\1}$ consists of weight vectors. Let 
\[
R:=\HH^{\bullet}(\ff ,\ff_{\0};\C ) \cong S^{\bullet}(\ff_{\1}^{*})=\C[\ff_{\1}]\cong \C [X_{1},X_{2},\dots, X_{m},Y_{1},Y_{2},\dots,Y_{m}],
\] where $X_j$ and $Y_j$ are defined by letting  $X_{j}: \ff_{\1} \longrightarrow \C$ be the linear functional given on our basis of matrix units for $\ff_{\1}$ by $X_{j}\left({e_{m-j+1, m+j}} \right)=1$ and otherwise zero. The functionals 
$Y_{j}: \ff_{\1} \longrightarrow \C$ are defined similarly. 

For our purposes it will suffice to consider the maximal ideal spectrum version of the support varieties. One has 
\[
V_{({\mathfrak f},{\mathfrak f}_{\0})}(\C)_{\max}\cong V^{r}_{\ff_{\1}}(\C ) = \operatorname{Proj}\left(\operatorname{MaxSpec}(R) \right) = \operatorname{Proj}(\ff_{\1}).
\]

If $N=\operatorname{Norm}_{G_{\0}}(\ff_{\1})$ then $N=\Sigma_{m}TL$ where $L$ fixes $R$ and acts trivially on $V_{({\mathfrak f},{\mathfrak f}_{\0})}(\C)_{\max}$. Here 
$\Sigma_{m}$ denote the symmetric group on $m$ letters embedded diagonally in $G_{\0}$.  Let $W$ be such a variety.  Since $\Sigma_{m}$ is a finite group we may write 
\[
W=\Sigma_{m}V
\] for some $T$-invariant closed subvariety $V$ of $\Proj (\ff_{\1})$. 

Let $Z(I)$ be the closed subvariety of $\Proj (\ff_{\1})$ determined by a homogeneous ideal $I$ of $R$. It follows that $V$ must be of the form   
\begin{equation}\label{E:generalV}
V= Z(X_{a_{1}}, \dotsc , X_{a_{s}}, Y_{b_{1}}, \dotsc , Y_{b_{t}}, g_{1}, \dotsc , g_{r})
\end{equation}
where $g_{1}, \dotsc , g_{r}$ are homogeneous polynomials of weight zero for $T$.

For the analysis in \cite{BKN4}, one employed the following conventions that will also be useful here. Let $s, t, p$ be non-negative integers with $m\geq s, t \geq p \geq 0$, and let
\begin{equation}\label{E:vanishingcoordinates}
V(s,t,p) = Z(X_{1},\dots,X_{s},Y_{s-p+1},\dots,Y_{s-p+t} ).
\end{equation}
This is the conical variety given by the vanishing of $s$ ``$X$ coordinates", $t$ ``$Y$ coordinates", with $p$ overlapping pairs.

\subsection{Examples:} In this section, we will use examples to illustrate the concepts from the prior sections. 

\subsubsection{$\mathfrak g=\mathfrak{sl}(1|1)$} In this case $G_{\0}=T$ is one-dimensional and consists of $2\times 2$ matrices which are non-zero scalar multiples of the identity matrix. The group 
$T$ acts trivally on $R=S^{\bullet}({\mathfrak f}_{\1}^{*})$. Therefore, the homogeneous polynomials in $R$ all have weight zero. This is in stark comparison to the $\mathfrak{gl}(1|1)$-case where the 
weight zero polynomials are the polynomials in the variable $Z_{1}=X_{1}Y_{1}$. 

\subsubsection{$\mathfrak g=\mathfrak{sl}(1|2)$} This case is more like the case for $\mathfrak{gl}(1|2)$ The torus $T$ has dimension 2, and the vectors $X_{1}$ and $Y_{1}$ are not fixed by $T$. The 
the weight zero polynomial are polynomials in the variable $Z_{1}=X_{1}Y_{1}$. 

\subsubsection{$\mathfrak g=\mathfrak{sl}(2|2)$} The maximal torus $T$ does not act trivially on $R$. A direct computation via the determinant condition shows that $R^{T}$ is generated by 
$X_{1}Y_{1}$, $X_{1}X_{2}$, $X_{2}Y_{2}$, and $Y_{1}Y_{2}$. One still needs to consider the varieties $V(s,t,p)$ is this case to describe the $N$-invariant subvarieties of ${\mathfrak f}_{\1}$. 
 
\subsection{}  The aim of this section is to outline how to realize every $N$-invariant closed subvariety of $V^{r}_{\ff_{\1}}(\C) \cong \Proj (\ff_{\1})$ as $V^{r}_{\ff_{\1}}(M)$ for some $M \in \FF_{(\fg,\fg_{\0})}$ in the 
case of $\mathfrak{sl}(m|n)$. The verification will use the results and the proofs in \cite{BKN4} which accomplish this for $\mathfrak{gl}(m|n)$. 

The first proposition can be proved in the same way as in in \cite[Proposition 7.2.1]{BKN4} via the construction of $L_{\zeta}$-modules from elements $\zeta$ in the cohomology ring. 

\begin{prop}  Let $g_{1},g_{2},\dots,g_{r}$ be homogeneous polynomials in $R$ of weight zero with respect to $T$. Then there exists a module $M$ in $\FF$ such that 
$V^{r}_{\ff_{\1}}(M)=\Sigma_{m} Z(g_{1},g_{2},\dots,g_{r})$. 
\end{prop} 

The next series of steps involves realization involving certain subvarieties $\Sigma_{r}V(s,t,p)$ using modules for $\mathfrak{gl}(m|n)$. These varieties can be realized for modules in $\mathfrak{sl}(m|n)$ by 
simply restricting the action of the modules realized for $\mathfrak{gl}(m|n)$. Finally, one can use the proof in \cite[Theorem 7.12.1]{BKN4}, to obtain the following result. 

\begin{theorem}\label{T:finalrealization}  Let $m \geq s,t \geq p \geq 0$ and let  $g_{1}, \dotsc , g_{r}$ be homogeneous weight zero polynomials.  Set $V=V(s,t,p) \cap Z (g_{1}, \dotsc , g_{r})$. Then there exists a finite dimensional  $\fg$-module $M$ such that 
\[
V^{r}_{\ff_{\1}}(M) = \Sigma_{m}V.
\]

\end{theorem}

\subsection{} For Lie superalgebras of Type A, we can now compute the Balmer spectrum, and verify the Nerves of Steel Conjecture. 

\begin{theorem} Let ${\mathfrak g}$ be $\mathfrak{gl}(m|n)$ or $\mathfrak{sl}(m|n)$, and ${\mathfrak f}\cong {\mathfrak f}_{\0}\oplus {\mathfrak f}_{\1}$ be a detecting subalgebra of 
${\mathfrak g}$. Then  
\begin{itemize}
\item[(a)] There exists a 1-1 correspondence 
$$
\{\text{thick tensor ideals of $\operatorname{stab}({\mathcal F}_{({\mathfrak g},{\mathfrak g}_{\0})})$}\} \begin{array}{c} {} \atop {\longrightarrow} \\ {\longleftarrow}\atop{} \end{array}  \XX_{sp}
$$
where $X=N\text{-}\operatorname{Proj}(S^{\bullet}({\mathfrak f}_{\1}))$. 
\item[(b)] There exists a homeomorphism $\eta: N$-$\operatorname{Proj}(S^{\bullet}({\mathfrak f}_{\1}))\longrightarrow \Spc(\operatorname{stab}{(\mathcal F}_{({\mathfrak g},{\mathfrak g}_{\0})}))$. 
\item[(c)] The comparison map $\phi:\operatorname{Spc}^{h}(\operatorname{stab}({\mathcal F}_{({\mathfrak g},{\mathfrak g}_{\0})}))\rightarrow 
\operatorname{Spc}(\operatorname{stab}({\mathcal F}_{({\mathfrak g},{\mathfrak g}_{\0})}))$ is bijective. 
\end{itemize}
\end{theorem}

\begin{proof} The statement of the theorem follows from Theorem~\ref{T:capstone} once Conditions (i) and (ii) are verified. For (i), in the cases when ${\mathfrak g}\cong \mathfrak{gl}(m|n)$ or $\mathfrak{sl}(m|n)$, the detecting subalgebras ${\mathfrak f}$ are splitting subalgebras.

For (ii), the realization property can be reduced to the case of the maximal ideal spectrum, i.e.,  one can realize every $N$-invariant closed subvariety of $V^{r}_{\ff_{\1}}(\C) \cong \Proj (\ff_{\1})$ as $V^{r}_{\ff_{\1}}(M)$ for some $M \in \FF_{(\fg,\fg_{\0})}$ (cf. \cite[Section 2.4]{BKN4}). For ${\mathfrak g}=\mathfrak{gl}(m|n)$, this holds by \cite[Theorem 7.21.1]{BKN4}, and for 
${\mathfrak g}=\mathfrak{sl}(m|n)$ by Theorem`\ref{T:finalrealization}. 
\end{proof} 

We remark that the Lie superalgebra $\mathfrak g=\mathfrak{psl}(n|n)$ is a quotient of $\mathfrak{sl}(n|n)$. One cannot simply restrict modules like in the case for $\mathfrak{gl}(m|n)$ to $\mathfrak{sl}(m|n)$. This 
means that new techniques need to be developed to verify Condition (ii) for $\mathfrak{psl}(n|n)$.

\end{document}